\documentclass[12pt,reqno]{amsart}
\usepackage{amssymb,amsmath,amsthm,enumerate,bbm}
\usepackage[a4paper]{geometry}
\usepackage{times}

\DeclareMathOperator{\Ran}{Ran}
\DeclareMathOperator{\Ker}{Ker}
\DeclareMathOperator{\Tr}{Tr}
\DeclareMathOperator{\sign}{sign}
\DeclareMathOperator{\const}{const}
\DeclareMathOperator{\Dom}{Dom}
\DeclareMathOperator{\Span}{Span}

\renewcommand\Re{\hbox{{\rm Re}}\,}
\newcommand{\abs}[1]{\lvert#1\rvert}
\newcommand{\Abs}[1]{\left\lvert#1\right\rvert}
\newcommand{\norm}[1]{\lVert#1\rVert}

\newcommand{\jap}[1]{\langle#1\rangle}

\newcommand{\bbR}{{\mathbb R}}

\newcommand{\bbZ}{{\mathbb Z}}

\newcommand{\R}{\mathbb{R}}

\newcommand{\Sch}{{\mathbf S}}

\newcommand{\calH}{{\mathcal H}}

\numberwithin{equation}{section}

\theoremstyle{plain}
\newtheorem{theorem}{\bf Theorem}[section]
\newtheorem*{theorem*}{Theorem 1.1$'$}
\newtheorem{lemma}[theorem]{\bf Lemma}
\newtheorem{proposition}[theorem]{\bf Proposition}

\newtheorem{open problem}[theorem]{\bf Open problem}

\theoremstyle{remark}
\newtheorem*{remark*}{\bf Remark}

\newcommand{\wt}{\widetilde}

\newcommand{\eps}{\varepsilon}

\newcommand{\Z}{\mathbb{Z}}

\begin{document}

\title[Eigenvalue clusters for hemisphere]{Eigenvalue clusters for the hemisphere Laplacian with variable Robin condition}

\author{Alexander Pushnitski}
\address{Department of Mathematics, King's College London, Strand, London, WC2R~2LS, U.K.}
\email{alexander.pushnitski@kcl.ac.uk}

\author{Igor Wigman}
\address{Department of Mathematics, King's College London, Strand, London, WC2R~2LS, U.K.}
\email{igor.wigman@kcl.ac.uk}

\begin{abstract}
We study the eigenvalue clusters of the Robin Laplacian on the 2-dimensional hemisphere with a variable Robin coefficient on the equator. The $\ell$'th cluster has $\ell+1$ eigenvalues.  We determine the asymptotic density of eigenvalues in the $\ell$'th cluster as $\ell$ tends to infinity. This density is given by an explicit integral involving the even part of the Robin coefficient.
\end{abstract}



\maketitle

\section{Introduction and main result}\label{sec.a}

\subsection{The Robin eigenvalues}
\label{sec:setup basic}
Let $\Omega$ be the two-dimensional upper unit hemisphere with its boundary $\partial \Omega$
that is the sphere equator.
As usual, we parameterise $\Omega$ by the spherical coordinates $\theta, \varphi$, where $\theta\in[0,\pi/2]$ is the polar angle (with the North pole corresponding to $\theta=0$, and the equator corresponding to $\theta=\pi/2$) and $\varphi\in(-\pi,\pi]$ is the azimuthal angle. Denote by $\Delta$ the Laplace-Beltrami operator on $\Omega$, which may be expressed in the spherical coordinates as
$$
\Delta=\frac1{\sin\theta}\frac{\partial}{\partial \theta}\sin\theta\frac{\partial}{\partial \theta}+\frac1{\sin^2\theta}\frac{\partial^2}{\partial\varphi^2}\ ,
$$
and let $\sigma=\sigma(\varphi)$ be a $C^1$-smooth real-valued function on the equator. We are interested in the eigenvalues $\lambda=\lambda(\sigma)$ of the \emph{Robin problem}
$$
-\Delta u=\lambda u \text{ on $\Omega$}, \qquad \frac{\partial u}{\partial n}+\sigma u=0 \text{ on $\partial\Omega$,}
$$
where $\frac{\partial u}{\partial n}=\frac{\partial u}{\partial \theta}$ is the normal derivative at the boundary.
We will call these the \emph{Robin eigenvalues} for short and enumerate them as
$$
\lambda_1(\sigma)\leq \lambda_2(\sigma)\leq\cdots
$$
listing them repeatedly in case of multiplicity. We recall that $\lambda_n(\sigma)\to\infty$ as $n\to\infty$.

The case $\sigma=0$ corresponds to the Neumann problem; this case allows for separation of variables. The corresponding eigenfunctions are spherical harmonics (those spherical harmonics that satisfy the Neumann boundary condition, see Section~\ref{sec.b2} below), and the eigenvalues are $\ell(\ell+1)$, $\ell\geq0$, of multiplicities $\ell+1$.
The case $\sigma=\text{const}$ also allows for separation of variables, and was considered by Rudnick and Wigman \cite{RW} (with an additional restriction $\sigma>0$).
Our first preliminary result is that the Robin eigenvalues form clusters of the size $O(\sqrt{\ell})$ around the Neumann eigenvalues $\ell(\ell+1)$.

\begin{theorem}\label{thm:a1}
There exists a constant $C=C(\sigma)>0$ such that all Robin eigenvalues belong to the union of intervals
$$
\bigcup\limits_{\ell=0}^{\infty} \Lambda_\ell, \qquad
\Lambda_\ell=\bigl(\ell(\ell+1)-C\sqrt{\ell+1},\ell(\ell+1)+C\sqrt{\ell+1}\bigr).
$$
Moreover, for all sufficiently large $\ell$, the total number of eigenvalues (counting multiplicities) in $\cup_{k=0}^{\ell-1}\Lambda_k$ is
\[
L:=\ell(\ell+1)/2,
\]
and there are exactly $\ell+1$ eigenvalues in $\Lambda_\ell$. Thus, for all sufficiently large $\ell$, the eigenvalues in $\Lambda_\ell$ are $\lambda_{L+k}(\sigma)$ with $k=1,\dots,\ell+1$.
\end{theorem}

Theorem~\ref{thm:a1} is proved in Section~\ref{sec:proof Lemma a1} below.
We will use the term \emph{$\ell$'th cluster} for the Robin eigenvalues in $\Lambda_\ell$.
The results of \cite{RW} for $\sigma=\text{const}>0$ assert that the estimate $O(\sqrt{\ell})$ on the size of the $\ell$'th cluster is sharp: there are Robin eigenvalues near $\ell(\ell+1)+C\sqrt{\ell}$ for large $\ell$.

\subsection{Main result}
Our aim is to determine the density of Robin eigenvalues in the $\ell$'th cluster. In other words, we are interested in the differences
$$
\lambda_{L+k}(\sigma)-\ell(\ell+1), \quad k=1,\dots,\ell+1
$$
for large $\ell$; these differences are termed the \emph{Robin-Neumann gaps} \cite{RWY}.
We set
$$
\sigma_\text{even}(\varphi)=\frac12(\sigma(\varphi)+\sigma(\varphi+\pi)).
$$
It turns out that, in the high energy limit, the density of Robin eigenvalues depends only on $\sigma_\text{even}$. Our main result is:
\begin{theorem}\label{thm.a2}
Let $f\in C^\infty(\bbR)$ be a compactly supported test function with $f(0)=0$. Then
\[
\lim_{\ell\to\infty}
\frac1{\ell+1}\sum_{k=1}^{\ell+1}f\bigl(\lambda_{L+k}(\sigma)-\ell(\ell+1)\bigr)
=
\frac1{4\pi}
\int_{-\pi}^{\pi}
\int_{-1}^1 f\left(\frac{4\sigma_\text{\rm even}(\varphi)}{\pi\sqrt{1-\xi^2}}\right)d\xi\,  d\varphi,
\]
where $L=\ell(\ell+1)/2$.
\end{theorem}

In fact, with some more effort one can prove that Theorem~\ref{thm.a2} holds true for all Lipschitz functions $f$, i.e. $\abs{f(x)-f(y)}\leq C\abs{x-y}$; we do not pursue this direction.

\subsection{Odd $\sigma$}

An application of Theorem \ref{thm.a2} with $\sigma$ odd, i.e. satisfying
$$
\sigma(\varphi+\pi)=-\sigma(\varphi), \quad \varphi\in(-\pi,\pi],
$$
implies that the distribution of the Robin-Neumann gaps in the $\ell$'th cluster converge, as $\ell \rightarrow\infty$, to the delta function at the origin. This leaves open the question of what happens for odd $\sigma$ after possible further rescaling. In this direction we prove the following result:

\begin{theorem}\label{thm.a3}
Let $\sigma$ be an odd real-valued trigonometric polynomial of degree $d$ (where $d$ is odd), viz. 
\begin{equation}
\sigma(\varphi)=\sum_{\genfrac{}{}{0pt}{1}{\abs{m}\leq d}{\text{$m$ \rm odd}}} 
c_m e^{im\varphi}, \quad \overline{c_{m}}=c_{-m}, \quad c_d\not=0.
\label{eq:def-odd-sigma}
\end{equation}
Then for all sufficiently large $\ell$, the Robin eigenvalues in the $\ell$'th cluster satisfy
$$
\lambda_{L+k}(\sigma)=\ell(\ell+1), \quad L=\ell(\ell+1)/2
$$
for all $k=1,\dots,\ell+1$ except possibly for $d+1$ indices $k$.
In other words, at most $d+1$ Robin-Neumann gaps in the $\ell$'th cluster are non-zero.
\end{theorem}
This shows that, at least for odd trigonometric polynomials, it is not possible to rescale the Robin-Neumann gaps to produce a meaningful limit.
The proof of Theorem \ref{thm.a3} is given in Section~\ref{sec:odd sigma}.

\subsection{Discussion}
By a change of variable, the statement of Theorem \ref{thm.a2} may be expressed as
\[
\lim_{\ell\to\infty}
\frac1{\ell+1}\sum_{k=1}^{\ell+1}f\bigl(\lambda_{L+k}(\sigma)-\ell(\ell+1)\bigr)
=
\int_{-\infty}^\infty f(y)\rho(\sigma;y)dy,
\]
where the density $\rho(\sigma;y)$ is given by (denoting $a_+=\max\{a,0\}$ and $a_-=\max\{-a,0\}$)
$$
\rho(\sigma;y)
=
\frac{1}{2\pi y^3}
\int_{-\pi}^\pi \left(\frac{4(\sigma_\text{\rm even}(\varphi))_+}{\pi}\right)^2 \left(1-\left(\frac{4(\sigma_\text{\rm even}(\varphi))_+}{\pi y}\right)^2\right)_+^{-1/2}d\varphi \quad \text{ for $y>0$}
$$
and
$$
\rho(\sigma;y)
=
\frac{1}{2\pi y^3}
\int_{-\pi}^\pi \left(\frac{4(\sigma_\text{\rm even}(\varphi))_-}{\pi}\right)^2 \left(1-\left(\frac{4(\sigma_\text{\rm even}(\varphi))_-}{\pi y}\right)^2\right)_+^{-1/2}d\varphi \quad \text{ for $y<0$.}
$$
If $\sigma$ is a positive constant, this gives
$$
\rho(\sigma;y)
=
\begin{cases}
\frac{16 \sigma^2}{\pi^2 y^3}
(1-(4\sigma/\pi y)^2)_+^{-1/2}, & y>0,
\\
0, & y\leq 0,
\end{cases}
$$
consistent with~\cite{RW}.

\vspace{2mm}

In the case of {\em constant} $\sigma$, the Robin problem allows for separation of variables.
For variable $\sigma$, this is no longer the case and we have to use some relatively advanced methods in spectral perturbation theory.

For $\sigma$ of {\em variable sign}, the density $\rho(\sigma;\cdot)$ is supported on both sides of the origin. As a consequence, for such $\sigma$ and for all sufficiently large $\ell$, there are Robin eigenvalues on both sides of $\ell(\ell+1)$ in the $\ell$'th cluster.

The assumed condition $\sigma\in C^1(\partial\Omega)$ could be relaxed to $\sigma\in L^\infty(\partial\Omega)$ and even to $\sigma\in L^p(\partial\Omega)$ for any $p>1$ at the expense of making the proofs more technical. This direction is not pursued within this paper.

\subsection{Related results}
In spectral theory, results of the type of Theorem \ref{thm.a2} originate from the classical work by A. Weinstein \cite{Weinstein} (see also \cite{G1,CdV}).
Weinstein considered the operator $-\Delta+V$, where $\Delta$ is the Laplace-Beltrami operator on an $n$-dimensional sphere (more generally, on a symmetric space of rank one) and $V$ is the operator of multiplication by a smooth real-valued function $V$ on the sphere (usually called a potential in this context). In this case, all of the eigenvalues of $-\Delta$ have finite multiplicities which grow with the eigenvalue. Adding the perturbation $V$ splits each eigenvalue into a cluster. Weinstein proved that the asymptotic eigenvalue distribution inside these clusters has a density function given by averaging $V$ along the closed geodesics of the sphere. For some follow-up work in mathematical physics see e.g. \cite{VB,RPVB}. Although these results seem related to our Theorem~\ref{thm.a2}, in the next subsection it is shown that a \emph{formal} application of Weinstein's formula in our setup produces a wrong density function due to an invalid interchange of limits w.r.t. different parameters.

Observe that Weinstein's formula for the density function depends only on the \emph{even} part of the potential $V$. In \cite{G3} Guillemin studied the eigenvalue clusters of $-\Delta+V$ on the sphere, where $V$ is an \emph{odd} potential. An application of Weinstein's formula for an odd potential yields a delta-function at the origin for the density function. Guillemin's result suggests that, after a suitable rescaling, one can still obtain a (different) density function, at least for some odd potentials. This scenario should be compared to our Theorem~\ref{thm.a3}.

\subsection{Comparison with Weinstein's formula}
Here we show what happens upon \emph{formally} applying Weinstein's formula in our setup. We will regard the Robin Laplacian on the hemisphere as a (singular) limit of the perturbed Neumann Laplacian. The geodesics
of the upper hemisphere $\Omega$ corresponding to the Neumann Laplacian are the great circles that are reflected from the equator according to Snell's law, i.e. (in this case) angle of incidence equals angle of refraction.

Let us parameterise the geodesics as follows. Each geodesic consists of two semi-circles, say $A_1$ and $A_2$, that meet at the equator. Let $\omega_1\in\Omega$ (resp. $\omega_2\in\Omega$) be the normal vector to
the plane containing $A_1$ (resp. $A_2$). If $\omega_1$ has the spherical coordinates $(\theta,\varphi)$, then $\omega_2$ has the spherical coordinates $(\theta,\varphi+\pi)$.
Let us denote the corresponding geodesic by $\Gamma(\theta,\varphi)$. It is clear that $\Gamma(\theta,\varphi+\pi)=\Gamma(\theta,\varphi)$. Thus, we will parameterise all the geodesics if we choose the point $(\theta,\varphi)$ to vary over half of the hemisphere: $\varphi\in[-\pi/2,\pi/2)$ and $\theta\in(0,\pi/2]$. Observe that the length of every such geodesic is $2\pi$. The geodesic $\Gamma(\theta,\varphi)$ gets reflected from the equator at the two points with the azimuthal angles $\varphi\pm\pi/2$, and the angle of incidence/refraction is $\theta$.

Now let us consider the operator $-\Delta_N+V$, where $\Delta_N$ is the Laplacian on the hemisphere with the Neumann boundary condition on the equator (this corresponds to the Robin condition with $\sigma=0$). We denote by
\[
\widetilde{V}(\theta,\varphi)=\frac1{2\pi}\int_{\Gamma(\theta,\varphi)}V(s)ds
\]
the average of $V$ over $\Gamma(\theta,\varphi)$; here $ds$ is the length element of the geodesic. Let $\{\lambda_k(V)\}_{k=1}^\infty$ be the eigenvalues of $-\Delta_N+V$ (listed in non-decreasing order).
These eigenvalues form clusters around the points $\ell(\ell+1)$. Formally applying Weinstein's formula for the density of eigenvalues to this case yields
\begin{equation}
\lim_{\ell\to\infty}\frac1{\ell+1}\sum_{k=1}^{\ell+1}f(\lambda_k(V)-\ell(\ell+1))
=
\frac1\pi \int_0^\pi \int_{0}^{\pi/2}f(\widetilde{V}(\theta,\varphi))\sin\theta\, d\theta\, d\varphi;
\label{a1}
\end{equation}
here the integral on the right hand side effects
averaging over all geodesics, and the factor $\pi$ in the denominator is the surface area of the quater-sphere (half of the upper hemisphere).

Now let $\sigma$ be a continuous function on the equator. For $\eps>0$ sufficiently small we define the function $V_\eps$ on the upper hemisphere by
\[
V_\eps(\theta,\varphi)
=
\begin{cases}
\frac1\eps\sigma(\varphi), & \text{if } \theta>\frac{\pi}{2}-\eps,
\\
0, & \text{if }\theta\leq\frac\pi2-\eps.
\end{cases}
\]
It is not difficult to see that the corresponding operator $-\Delta_N+V_\eps$ converges (in the strong resolvent sense) to the Robin Laplacian with the Robin parameter $\sigma$. Let us apply Weinstein's formula \eqref{a1} to $V_\eps$ and \emph{formally} interchange the limits w.r.t. $\ell\to\infty$ and $\eps\to0_+$. We claim that the resulting formula for the density function differs from the one in Theorem~\ref{thm.a2} by a factor of $2$. Indeed, simple geometric considerations show that averaging over a geodesic $\Gamma(\theta,\varphi)$ for sufficiently small $\eps$ yields
\[
\lim_{\eps\to0_+}\widetilde{V}_\eps(\theta,\varphi)
=\frac{2\sigma(\varphi+\tfrac{\pi}{2})+2\sigma(\varphi-\tfrac{\pi}{2})}{2\pi\sin\theta}
=\frac{2\sigma_{\mathrm{even}}(\varphi+\tfrac{\pi}{2})}{\pi\sin\theta}
\]
and therefore
\begin{align*}
\lim_{\eps\to0_+}
&\frac1\pi \int_0^\pi \int_{0}^{\pi/2}f(\widetilde{V}_\eps(\theta,\varphi))\sin\theta\, d\theta\, d\varphi
=
\frac1\pi \int_0^\pi \int_{0}^{\pi/2}f\left(\frac{2\sigma_{\mathrm{even}}(\varphi+\tfrac{\pi}{2})}{\pi\sin\theta}\right)\sin\theta\, d\theta\, d\varphi
\\
&=\frac1\pi \int_0^\pi \int_{0}^{1}f\left(\frac{2\sigma_{\mathrm{even}}(\varphi+\tfrac{\pi}{2})}{\pi\sqrt{1-\xi^2}}\right)\, d\xi\, d\varphi
=\frac1\pi \int_0^\pi \int_{0}^{1}f\left(\frac{2\sigma_{\mathrm{even}}(\varphi)}{\pi\sqrt{1-\xi^2}}\right)\, d\xi\, d\varphi
\\
&=\frac1{4\pi} \int_{-\pi}^\pi \int_{-1}^{1}f\left(\frac{2\sigma_{\mathrm{even}}(\varphi)}{\pi\sqrt{1-\xi^2}}\right)\, d\xi\, d\varphi,
\end{align*}
which differs from the correct result by a factor of 2 in the argument of $f$. We conclude that the ``naive'' application of Weinstein's formula in this case produces an incorrect result.

We observe that a similar phenomenon can be seen in the simple one-dimensional case.
Indeed, consider the Sturm-Liouville problem on the interval $[0,1]$.
The eigenvalues $\lambda_n(\sigma)$ of the Robin problem
\[
-u''=\lambda u, \quad -u'(0)+\sigma u(0)=u'(1)=0
\]
obey
\[
\lambda_n(\sigma)-\pi^2 n^2=2\sigma+o(1), \quad n\to\infty,
\]
while the eigenvalues $\mu_n(\sigma,\eps)$ of the problem
\[
-u''+\sigma\cdot \eps^{-1}\chi_{(0,\eps)}u=\mu u, \quad u'(0)=u'(1)=0
\]
obey
\[
\mu_n(\sigma,\eps)-\pi^2 n^2=\sigma+o(1), \quad n\to\infty
\]
for any $\eps>0$. Even though $\mu_n(\sigma,\eps)\to\lambda_n(\sigma)$ as $\eps\to0$ for each $n$, interchanging the limits $n\to\infty$ and $\eps\to0$ produces an incorrect result.

\subsection*{Acknowledgements}

The authors are grateful to Ze\'ev Rudnick for useful remarks on the preliminary version of the text, and to Jean Lagac\'e for useful discussions.
It is a pleasure to thank the anonymous referees for a thorough reading of the manuscript, that, in particular, helped us to rectify a few technical issues in the previous version of this paper.

\section{Definitions and key steps of the proof}\label{sec.bb}

\subsection{Summary}

We consider the Robin Laplacian as the perturbation of the Neumann Laplacian. The Robin-Neumann gaps in the $\ell$'th cluster are identified (up to a small error term) with the eigenvalues of a certain operator $V_\ell[\sigma]$ acting in the $\ell$'th eigenspace of the Neumann Laplacian; this is Theorem~\ref{thm.d2} below. Next, $V_\ell[\sigma]$ is shown to be unitarily equivalent to a semiclassical pseudodifferential operator acting
in the $L^2$ space of the equator $\partial\Omega$, with the semiclassical parameter $1/\ell$ (again up to a small error term). Finally, we appeal to standard trace asymptotics for  semiclassical operators, which yields the result.

\subsection{Notation and preliminaries}
We will work with operators in $L^2(\Omega)$, which is the Hilbert space with the norm
$$
\norm{u}^2_{L^2(\Omega)}
=
\int_{-\pi}^\pi\int_{0}^{\pi/2} \abs{u(\theta,\varphi)}^2\sin\theta\,  d\theta\,  d\varphi,
$$
and in $L^2(\partial\Omega)$, which is the Hilbert space with the norm
$$
\norm{u}^2_{L^2(\partial\Omega)}
=
\int_{-\pi}^\pi\abs{u(\varphi)}^2  d\varphi.
$$
We denote the corresponding inner products (linear in the first entry and anti-linear in the second one) by $\jap{\cdot,\cdot}_{L^2(\Omega)}$ and $\jap{\cdot,\cdot}_{L^2(\partial\Omega)}$.

If $A$ is a compact operator in a Hilbert space, let $\{s_n(A)\}_{n=1}^\infty$ be the sequence of singular values of $A$ (listed repeatedly in case of multiplicities).
For $1\leq p<\infty$, below
$$
\norm{A}_{\Sch_p}
=
\left(\sum_n s_n(A)^p\right)^{1/p}
$$
is the norm of $A$ in the Schatten class $\Sch_p$.
In fact, we will only need the trace norm $\norm{A}_{\Sch_1}$ and the Hilbert-Schmidt norm $\norm{A}_{\Sch_2}$. We recall the analogue of the Cauchy-Schwarz inequality for the Schatten norms,
$$
\norm{AB}_{\Sch_1}\leq\norm{A}_{\Sch_2}\norm{B}_{\Sch_2}.
$$
We also recall that the operator trace is a bounded linear functional on the trace class, with $\abs{\Tr A}\leq \norm{A}_{\Sch_1}$.

\subsection{Closed semibounded quadratic forms and the variational principle}\label{sec.var}
We briefly recall the notions related to the variational principle; see e.g. \cite[Section 10.2.3]{BS} for the details. Let $A$ be a self-adjoint operator with the domain $\Dom A$ on a Hilbert space. We recall that $A$ is called lower semi-bounded if for any $f\in\Dom A$ we have
$$
\jap{Af,f}\geq m\norm{f}^2
$$
with some constant $m\in\bbR$. By adding a multiple of the identity operator to $A$ if necessary, we may always assume that $m>0$. The \emph{quadratic form of $A$} is $a(f)=\norm{A^{1/2}f}^2$ with the domain $\Dom a=\Dom A^{1/2}$ (where the square root is defined in the sense of the functional calculus for self-adjoint operators). One can prove that the quadratic form $a$ is closed, i.e.  $\Dom a$ is complete with respect to the norm generated by $a$. Conversely, every closed lower semi-bounded quadratic form corresponds to a unique self-adjoint operator.

Let $A$ and $B$ be self-adjoint lower semi-bounded operators, and let $a$ and $b$ be the corresponding quadratic forms with the domains $\Dom a$ and $\Dom b$. One writes $A\leq B$ if $\Dom b\subset\Dom a$ and
$$
a(f)\leq b(f), \quad \forall f\in\Dom b.
$$
The utility of this notion comes from the variational principle. One of the versions of this principle asserts that if both $A$ and $B$ have compact resolvents, then the sequences of eigenvalues of $A$ and $B$ (enumerated in non-decreasing order) satisfy
$$
\lambda_n(A)\leq \lambda_n(B)
$$
for all indices $n$.
See e.g. \cite[Theorem 10.2.4]{BS} for the details.

\subsection{The Robin Laplacian $H[\sigma]$}\label{sec.2.4}
For a smooth function $u$ on the hemisphere, we write
$$
\int_{\Omega}\abs{\nabla u}^2dx
:=
\int_{-\pi}^\pi\int_{0}^{\pi/2}
\left(\Abs{\frac{\partial u}{\partial \theta}}^2+\frac1{(\sin\theta)^2}\Abs{\frac{\partial u}{\partial \varphi}}^2\right)\sin\theta\,  d\theta\,  d\varphi
$$
for the Dirichlet integral. Recall that the Sobolev space $W^1_2(\Omega)$ is the set of all functions $u\in L^2(\Omega)$ such that the Dirichlet integral is finite. Furthermore, by the Sobolev Trace Theorem (see e.g. \cite[Theorem 1.1.2]{Necas}) functions $u\in W^1_2(\Omega)$ can be restricted to the equator; the restrictions belong to $L^2(\partial\Omega)$ and satisfy the estimate
\begin{equation}
\int_{-\pi}^\pi \abs{u(\pi/2,\varphi)}^2d\varphi\leq C\int_{\Omega}\abs{\nabla u}^2dx, \quad u\in W^1_2(\Omega),
\label{b10}
\end{equation}
with some absolute constant $C$. Moreover, the corresponding embedding operator $W^1_2(\Omega)\subset L^2(\partial\Omega)$ is compact; see e.g. \cite[Theorem 2.6.2]{Necas}.

Let $h[\sigma]$ be the quadratic form
$$
h[\sigma](u)=\int_{\Omega}\abs{\nabla u}^2dx+\int_{-\pi}^\pi\sigma(\varphi)\abs{u(\pi/2,\varphi)}^2d\varphi,
\quad
u\in W^1_2(\Omega).
$$
Since $\sigma$ is bounded, by the embedding \eqref{b10}, the second integral in the right hand side is well-defined.
The form $h[\sigma]$ is closed in $L^2(\Omega)$; we denote by $H[\sigma]$ the self-adjoint operator corresponding to this form. This operator is the Robin Laplacian, whose eigenvalues have been earlier denoted by $\lambda_n(\sigma)$. Note that the case $\sigma=0$ corresponds to the Neumann Laplacian $H[0]$, and so we can write
\begin{equation}
h[\sigma](u)=h[0](u)+\int_{-\pi}^\pi\sigma(\varphi)\abs{u(\pi/2,\varphi)}^2d\varphi.
\label{d0}
\end{equation}
Below we consider $H[\sigma]$ as a perturbation of $H[0]$ in the quadratic form sense.
An important technical point here is that (by the compactness of the embedding $W^1_2(\Omega)\subset L^2(\partial\Omega)$) the quadratic form given by the integral over the equator in \eqref{d0} is compact in $W^1_2(\Omega)$, and therefore $H[\sigma]$ can be considered as an operator obtained from $H[0]$ by adding a relatively form-compact perturbation; see e.g. \cite{RS4} for the discussion of the concept of form-compactness.

\subsection{Spherical harmonics}\label{sec.b2}
The solutions to the eigenvalue equation
$$
-\Delta Y=\ell(\ell+1) Y
$$
on the sphere are the spherical harmonics of degree $\ell$; these are the restrictions of the degree-$\ell$ harmonic polynomials on $\R^{3}$ to the sphere. The linear space of degree-$\ell$ spherical harmonics is spanned by the $(2\ell+1)$ functions
\begin{equation}
Y_\ell^m(\theta,\varphi)
=
\sqrt{\frac{(2\ell+1)}{2\pi}\frac{(\ell-m)!}{(\ell+m)!}}P_\ell^m(\cos\theta)e^{im\varphi},
\quad
m=-\ell,\dots,\ell,
\label{b11}
\end{equation}
where $P_\ell^m$ are the associated Legendre polynomials, and the normalisation is chosen so that $Y_\ell^m$ are orthonormal in $L^2(\Omega)$:
$$
\int_0^{\pi/2} \int_{-\pi}^\pi Y_\ell^m(\theta,\varphi) \overline{Y_\ell^{m'}(\theta,\varphi)}\sin \theta\,  d\theta\, d\varphi=\delta_{m,m'}
$$
(this differs from the usual normalisation by the factor of $\sqrt{2}$ because the usual normalisation corresponds to the full sphere).
The spherical harmonics with $m-\ell$ even satisfy the Neumann boundary condition at the equator, whereas the ones with $m-\ell$ odd satisfy the Dirichlet boundary condition.

\subsection{Neumann eigenspaces}
\label{sec:Pl def}

We are interested in the Neumann eigenfunctions. Let $\mathbf P_\ell$ be the orthogonal projection in $L^2(\Omega)$ onto the eigenspace of the Neumann Laplacian corresponding to the eigenvalue $\ell(\ell+1)$. In other words, $\mathbf P_\ell$ is the projection onto the subspace spanned by $Y_\ell^m$, $m=-\ell,-\ell+2,\dots,\ell-2,\ell$:
\begin{equation}
\mathbf P_\ell:
u(\theta,\varphi)\mapsto
\sum_{\genfrac{}{}{0pt}{1}{m=-\ell}{\text{$m-\ell$ even}}}^\ell Y_\ell^m(\theta,\varphi)\int_0^{\pi/2}\int_{-\pi}^\pi \overline{Y_\ell^m(\theta',\varphi')}u(\theta',\varphi')\sin \theta'\, d\theta'\  d\varphi'\ .
\label{b11a}
\end{equation}
Clearly, the dimension of the range $\Ran \mathbf P_\ell$ is $\ell+1$.

\subsection{The operator $V_\ell[\sigma]$}
For $\ell\geq0$, let $V_\ell[\sigma]$ be the operator in $\Ran {\mathbf P}_\ell$ corresponding to the quadratic form
\begin{equation}
v_\ell[\sigma](u)=
\int_{-\pi}^\pi \sigma(\varphi)\abs{u(\pi/2,\varphi)}^2d\varphi, \quad u\in\Ran {\mathbf P}_\ell.
\label{b8}
\end{equation}
We note that for any spherical harmonic, the restriction onto the equator $\theta=\pi/2$ is a continuous function, and so the above integral is well-defined.

Furthermore, by \eqref{b11}, for any function $u\in\Ran {\mathbf P}_\ell$, the restriction onto the equator can be written as
$$
u(\pi/2,\varphi)=e^{i\ell\varphi}w(\varphi),
$$
where $w$ is even, i.e.  $w(\varphi+\pi)=w(\varphi)$, since $m-\ell$ in \eqref{b11a} is even.
It follows that $\abs{u(\pi/2,\varphi)}^2$ is even.
Therefore, $\sigma$ in \eqref{b8} can be replaced by $\sigma_\text{even}$, and so
$$
V_\ell[\sigma]=V_\ell[\sigma_\text{even}].
$$
This explains why the dependence of the limiting density of eigenvalues on $\sigma$ in Theorem~\ref{thm.a2} is through $\sigma_\text{even}$ alone.
The first key ingredient of our construction is the following result:
\begin{theorem}\label{thm.b1}
For any compactly supported test function $f\in C^\infty(\bbR)$ with $f(0)=0$, one has
\begin{equation}
\lim_{\ell\to\infty}\frac1{\ell+1}\Tr f(V_\ell[\sigma])
=
\frac1{4\pi}
\int_{-\pi}^{\pi}
\int_{-1}^1 f\left(\frac{4\sigma_\text{\rm even}(\varphi)}{\pi\sqrt{1-\xi^2}}\right)d\xi\,  d\varphi.
\label{eq.extra}
\end{equation}
Furthermore, \eqref{eq.extra} holds true with $f(x)=x$.
\end{theorem}
The proof of Theorem \ref{thm.b1} is given in Section~\ref{sec.b}.

\subsection{Upper and lower bounds on eigenvalues in the $\ell$'th cluster}
The second key ingredient of our construction is the following result. Below we denote by $\lambda_k(V_\ell[\sigma])$ the eigenvalues of $V_\ell[\sigma]$, enumerated in non-decreasing order.
Since the dimension of $\Ran \mathbf P_\ell$ is $\ell+1$, there are $\ell+1$ of these eigenvalues.
\begin{theorem}\label{thm.d2}
Let $\eps>0$, and let $\sigma_\pm=\sigma\pm\eps\abs{\sigma}$. Then for all sufficiently large $\ell$, we have the estimates
$$
\lambda_k(V_\ell[\sigma_-])
\leq
\lambda_{L+k}(\sigma)-\ell(\ell+1)
\leq
\lambda_k(V_\ell[\sigma_+]), \quad L=\ell(\ell+1)/2,
$$
for $k=1,\dots,\ell+1$.
\end{theorem}
The proof is given in Sections~\ref{sec.cc} and \ref{sec.d}.

\subsection{Proof of Theorem~\ref{thm.a2}}\label{sec.b6}
The proof is a simple combination of the upper and lower bounds of Theorem~\ref{thm.d2} and the asymptotics of Theorem~\ref{thm.b1}. Applying the estimate
$$
\abs{f(a)-f(b)}\leq \norm{f'}_{L^\infty}\abs{a-b}
$$
to the second inequality (upper bound) in Theorem~\ref{thm.d2} and summing over $k$, we find
\begin{align*}
\biggl|\sum_{k=1}^{\ell+1}f(\lambda_{L+k}(\sigma)-\ell(\ell+1))&-\sum_{k=1}^{\ell+1}f(\lambda_k(V_\ell[\sigma_+]))\biggr|
\\
&\leq
\norm{f'}_{L^\infty}\sum_{k=1}^{\ell+1}
\biggl(
\lambda_k(V_\ell[\sigma_+])-\lambda_{L+k}(\sigma)+\ell(\ell+1)
\biggr)
\\
&\leq
\norm{f'}_{L^\infty}\sum_{k=1}^{\ell+1}\biggl(\lambda_k(V_\ell[\sigma_+])-\lambda_k(V_\ell[\sigma_-])\biggr)
\\
&=
\norm{f'}_{L^\infty}\Tr V_\ell[\sigma_+-\sigma_-]
=
2\eps\norm{f'}_{L^\infty}\Tr V_\ell[\abs{\sigma}].
\end{align*}
Using Theorem~\ref{thm.b1} with $f(x)=x$, we find
$$
\Tr V_\ell[\abs{\sigma}]\leq C(\ell+1)
$$
with some $C=C(\sigma)$. We conclude that
\begin{align*}
\limsup_{\ell\to\infty}\frac1{\ell+1}
\sum_{k=1}^{\ell+1}f(\lambda_{L+k}(\sigma)-\ell(\ell+1))
\leq
\limsup_{\ell\to\infty}\frac1{\ell+1}
\sum_{k=1}^{\ell+1}f(\lambda_k(V_\ell[\sigma_+]))
+2C\eps.
\end{align*}

Using Theorem~\ref{thm.b1}, the upper limit in the right hand side can be computed, which gives
\begin{align*}
\limsup_{\ell\to\infty}\frac1{\ell+1}&
\sum_{k=1}^{\ell+1}f(\lambda_{L+k}(\sigma)-\ell(\ell+1))
\\
&\leq
\frac1{4\pi}
\int_{-\pi}^{\pi}
\int_{-1}^1 f\left(\frac{4\sigma_\text{\rm even}(\varphi)+4\eps\abs{\sigma_\text{\rm even}(\varphi)}}{\pi\sqrt{1-\xi^2}}\right)d\xi\,  d\varphi
+C\eps.
\end{align*}
Since $f$ is smooth, we find
$$
\int_{-\pi}^{\pi}
\int_{-1}^1 f\left(\frac{4\sigma_\text{\rm even}(\varphi)+4\eps\abs{\sigma_\text{\rm even}(\varphi)}}{\pi\sqrt{1-\xi^2}}\right)d\xi\,  d\varphi
=
\int_{-\pi}^{\pi}
\int_{-1}^1 f\left(\frac{4\sigma_\text{\rm even}(\varphi)}{\pi\sqrt{1-\xi^2}}\right)d\xi\,  d\varphi
+O(\eps)
$$
as $\eps\to0$, and so finally we obtain
\begin{align*}
\limsup_{\ell\to\infty}\frac1{\ell+1}&
\sum_{k=1}^{\ell+1}f(\lambda_{L+k}(\sigma)-\ell(\ell+1))
\\
&\leq
\frac1{4\pi}
\int_{-\pi}^{\pi}
\int_{-1}^1 f\left(\frac{4\sigma_\text{\rm even}(\varphi)}{\pi\sqrt{1-\xi^2}}\right)d\xi\,  d\varphi
+C'\eps
\end{align*}
with some constant $C'$.
In the same way we find
\begin{align*}
\liminf_{\ell\to\infty}\frac1{\ell+1}&
\sum_{k=1}^{\ell+1}f(\lambda_{L+k}(\sigma)-\ell(\ell+1))
\\
&\geq
\frac1{4\pi}
\int_{-\pi}^{\pi}
\int_{-1}^1 f\left(\frac{4\sigma_\text{\rm even}(\varphi)}{\pi\sqrt{1-\xi^2}}\right)d\xi\,  d\varphi
-C'\eps.
\end{align*}
As $\eps>0$ can be taken arbitrarily small, we finally obtain
\begin{align*}
&\limsup_{\ell\to\infty}\frac1{\ell+1}
\sum_{k=1}^{\ell+1}f(\lambda_{L+k}(\sigma)-\ell(\ell+1))
\\
=&
\liminf_{\ell\to\infty}\frac1{\ell+1}
\sum_{k=1}^{\ell+1}f(\lambda_{L+k}(\sigma)-\ell(\ell+1))
\\
=&
\frac1{4\pi}
\int_{-\pi}^{\pi}
\int_{-1}^1 f\left(\frac{4\sigma_\text{\rm even}(\varphi)}{\pi\sqrt{1-\xi^2}}\right)d\xi\,  d\varphi,
\end{align*}
which is the required statement.
\qed
\subsection{The structure of the rest of the paper}
In the rest of the paper, we prove Theorems~\ref{thm.b1} and \ref{thm.d2}.
In Section~\ref{sec.b}, we prove Theorem~\ref{thm.b1}. The proof of one well-known technical ingredient is postponed to Appendix A. In Section~\ref{sec.cc}, we prove some auxiliary estimates for the resolvent of $H[\sigma]$ and related operators.
In Section~\ref{sec.d}, we prove Theorem~\ref{thm.d2}, and also Theorem~\ref{thm:a1}.
In Section~\ref{sec:odd sigma} we prove Theorem~\ref{thm.a3}.

\section{The eigenvalues of $V_\ell[\sigma]$}\label{sec.b}

\subsection{Overview}
Our main aim in this section is to prove Theorem~\ref{thm.b1} and also the following simple statement:
\begin{lemma}\label{lma.b1}
We have the estimate
\begin{equation}
\norm{V_\ell[\sigma]}=O(\sqrt{\ell}), \quad \ell\to\infty.
\label{b9}
\end{equation}
\end{lemma}
By taking $\sigma=\const>0$, it is easy to check that this estimate is sharp, i.e. in this case
$\norm{V_\ell[\sigma]}\geq c\sqrt{\ell}$
for some $c>0$ and all sufficiently large $\ell$. Observe that Theorem~\ref{thm.b1} with $f(x)=x$ states
$$
\Tr V_\ell[\sigma]=O(\ell), \quad \ell\to\infty;
$$
it is instructive to compare this with \eqref{b9}.

\subsection{Multiplications and convolutions}
We need some notation. Below we consider operators acting in $L^2(\partial\Omega)$.
For a bounded function $a$ on  $\partial\Omega$, we denote by $M[a]$ the operator of multiplication by $a(\varphi)$ in $L^2(\partial\Omega)$, and by $C[a]$ the operator of convolution with $a$:
\begin{equation}
\label{eq:C[a] conv op def}
C[a]: u(\varphi)\mapsto \frac1{2\pi}\int_{-\pi}^\pi u(\varphi')a(\varphi-\varphi')d\varphi'.
\end{equation}
Recall that if $a(\varphi)=\sum_n a_ne^{in\varphi}$, then by Parseval's theorem,  $C[a]$ is unitarily equivalent to the operator of multiplication by the sequence $\{a_n\}$ in $\ell^2(\bbZ)$. In particular, for the operator norm and the Hilbert-Schmidt norm of $C[a]$ we have
\begin{align}
\norm{C[a]}&=\sup_n\abs{a_n},
\label{b0a}
\\
\norm{C[a]}_{\Sch_2}^2&=\sum_{n=-\infty}^\infty\abs{a_n}^2.
\label{b0b}
\end{align}

\subsection{Three functions on the equator}
Below we will make use of three trigonometric polynomials on the equator.
In order to define them, we first recall that by \eqref{b11}, for $m-\ell$ even, the restrictions of the spherical harmonics onto the equator are
\begin{equation}
\label{eq:restr equator Alm1}
Y_\ell^m(\pi/2,\varphi)=(-1)^{\frac{\ell+m}{2}}
\frac{A_{\ell,m}}{\sqrt{2\pi}}e^{im\varphi},
\end{equation}
where
\begin{equation}
\label{eq:restr equator Alm2}
A_{\ell,m}:=(-1)^{\frac{\ell+m}{2}}
\sqrt{(2\ell+1)\cdot \frac{(\ell-m)!}{(\ell+m)!}}P_\ell^m(0).
\end{equation}
The factor $(-1)^{\frac{\ell+m}{2}}$ is introduced here because (as we will see below) this makes the constants $A_{\ell,m}$ positive. It will be convenient to set $A_{\ell,m}=0$ if $\ell-m$ is odd or $\abs{m}>\ell$. Now let us denote
\begin{align}
x_\ell(\varphi)&=\sqrt{2\pi}\sum_{\genfrac{}{}{0pt}{1}{m=-\ell}{\text{$m-\ell$ even}}}^\ell (-1)^{\frac{\ell+m}{2}}Y_\ell^m(\pi/2,\varphi)
=
\sum_{\genfrac{}{}{0pt}{1}{m=-\ell}{\text{$m-\ell$ even}}}^\ell
A_{\ell,m}e^{im\varphi},
\notag
\\
\label{eq:y ell def}
y_\ell(\varphi)&=\sum_{\genfrac{}{}{0pt}{1}{m=-\ell}{\text{$m-\ell$ even}}}^\ell A_{\ell,m}^2 e^{im\varphi},
\\
z_\ell(\varphi)
&=
\sum_{\genfrac{}{}{0pt}{1}{m=-\ell+2}{\text{$m-\ell$ even}}}^{\ell-2}
\frac{2}{\sqrt{\pi}}(1-(m/\ell)^2)^{-1/4}e^{im\varphi}.
\label{b6}
\end{align}
We will show that $z_\ell$ is the asymptotic form of $x_\ell$  for large $\ell$.

\subsection{Reduction to an operator in $L^2(\partial\Omega)$}

Recall that $V_\ell[\sigma]$ is the operator acting in the $(\ell+1)$-dimensional subspace $\Ran {\mathbf P}_\ell\subset L^2(\Omega)$.
Below we reduce the analysis of the operator $V_\ell[\sigma]$ to the analysis of a simple operator in $L^2(\partial\Omega)$.
\begin{lemma}
The non-zero eigenvalues of $V_\ell[\sigma]$ coincide (including multiplicities) with the non-zero eigenvalues of the operator
\begin{equation}
W_\ell[\sigma]:=C[x_\ell]M[\sigma]C[x_\ell]\quad \text{ in $L^2(\partial\Omega)$.}
\label{b2a}
\end{equation}
In particular, for every test function $f\in C^\infty(\bbR)$ vanishing at the origin,
\begin{equation}
\Tr f(V_\ell[\sigma])=\Tr f(W_\ell[\sigma]).
\label{b2}
\end{equation}
\end{lemma}
\begin{proof}
Denote by
$e_m(\varphi)=\frac1{\sqrt{2\pi}}e^{im\varphi}$, $m\in\bbZ$,
the standard orthonormal basis in $L^2(\partial\Omega)$.
From \eqref{eq:restr equator Alm1}, \eqref{eq:restr equator Alm2},  it follows that
$$
C[x_\ell]e_m=\begin{cases}
(-1)^{\frac{\ell+m}{2}}Y_\ell^m(\pi/2,\cdot) & \abs{m}\leq\ell, \, \ell-m\text{ even} \\
0 & \text{otherwise}
\end{cases}
$$
By the definition of $x_\ell$, the range of the operator $W_\ell[\sigma]$ belongs to the $(\ell+1)$-dimensional subspace of $L^2(\partial\Omega)$ spanned by the elements
$e_m$, $m=-\ell,\dots,\ell$, with $m-\ell$ even.
Let us compute the matrix of $W_\ell[\sigma]$ in the basis $(-1)^{\frac{\ell+m}{2}}e_m$:
\begin{align*}
(-1)^{\frac{\ell+m}{2}}(-1)^{\frac{\ell+m'}{2}}
&\jap{W_\ell[\sigma]e_m,e_{m'}}_{L^2(\partial\Omega)}
\\
&=
(-1)^{\frac{\ell+m}{2}}(-1)^{\frac{\ell+m'}{2}}
\jap{M[\sigma]C[x_\ell]e_m,C[x_\ell]e_{m'}}_{L^2(\partial\Omega)}
\\
&=
\int_{-\pi}^\pi \sigma(\varphi) Y_\ell^m(\pi/2,\varphi)\overline{Y_\ell^{m'}(\pi/2,\varphi)}d\varphi .
\end{align*}
Next, let us compute the matrix of $V_\ell[\sigma]$ in the orthonormal basis of $\Ran {\mathbf P}_\ell$ given by $Y_\ell^m$, $m=-\ell,\dots,\ell$, with $m-\ell$ even:
\[
\jap{V_\ell[\sigma]Y_\ell^m,Y_\ell^{m'}}_{L^2(\Omega)}
=
\int_{-\pi}^\pi \sigma(\varphi)Y_\ell^m(\pi/2,\varphi)\overline{Y_\ell^{m'}(\pi/2,\varphi)}d\varphi.
\]
Comparing this, we see that the matrices of $V_\ell[\sigma]$ and $W_\ell[\sigma]$ coincide, hence the result.
\end{proof}

\subsection{Asymptotics of $x_\ell$}

Our next step is to replace $x_\ell$ by its asymptotics as $\ell\to\infty$.
We recall the constants $A_{\ell,m}$ defined in \eqref{eq:restr equator Alm2}. We have
$$
P_\ell^m(0)=(-1)^{(m+\ell)/2}\frac{2^m}{\sqrt{\pi}}\frac{\Gamma\left(\frac{\ell+m+1}{2}\right)}{\Gamma\left(\frac{\ell-m}{2}+1\right)},
$$
and therefore
$$
A_{\ell,m}=
2^m\sqrt{\frac{(2\ell+1)}{\pi}\frac{(\ell-m)!}{(\ell+m)!}}\frac{\Gamma\left(\frac{\ell+m+1}{2}\right)}{\Gamma\left(\frac{\ell-m}{2}+1\right)}>0.
$$
Below we prove that
$$
A_{\ell,m}^2=\frac{4}{\pi}(1-(m/\ell)^2)^{-1/2}+\text{error term}
$$
as $\ell\to\infty$ and $\ell^2-m^2\to\infty$,
with suitable estimates for the error term.
This is the key technical ingredient of the proof of Theorem~\ref{thm.b1}.

\begin{lemma}\label{lma.b2}
As $\ell\to\infty$, we have
\begin{align}
\sum_
{m=-\ell+1}^{\ell-1}&
(1-(m/\ell)^2)^{-\frac12}=O(\ell),
\label{b3}
\\
\sum_{m=-\ell}
^\ell& A_{\ell,m}^2
=O(\ell),
\label{b4}
\\
\sum_
{\genfrac{}{}{0pt}{1}{m=-\ell+1}{\text{$m-\ell$ even}}}
^{\ell-1}&
\Abs{A_{\ell,m}^2-\frac{4}{\pi}(1-(m/\ell)^2)^{-\frac12}}=O(\ell^{2/3}),
\label{b5}
\\
\sup_
{\abs{m}\leq \ell}
&A_{\ell,m}^2=O(\sqrt{\ell}).
\label{b5a}
\end{align}
\end{lemma}
\begin{proof}
Using the duplication formula for the Gamma function,
$$
\Gamma(z)\Gamma(z+\tfrac12)=\sqrt{\pi}2^{1-2z}\Gamma(2z),
$$
and assuming that $|m|\le \ell$ and $\ell-m$ is even, we can rewrite the expression for $A_{\ell,m}$ as
\begin{equation}
A_{\ell,m}^2=\frac{2\ell+1}{\pi}\gamma(\ell-m)\gamma(\ell+m),
\quad\text{ where }\quad
\gamma(x):=\frac{\Gamma(\frac{x}{2}+\frac12)}{\Gamma(\frac{x}{2}+1)}.
\label{b1}
\end{equation}
Recall that, by convention, for all other $\ell,m$, one has $A_{\ell,m}=0$.
By \eqref{b1}, we have $A_{\ell,-m}^2=A_{\ell,m}^2$ and so instead of summing from $-\ell$ to $\ell$ we can estimate the sums from $m=0$ to $m=\ell$.

The estimate \eqref{b3} is elementary:
\begin{equation}
\sum_{m=0}^{\ell-1}(1-(m/\ell)^2)^{-1/2}
\leq
\int_0^\ell(1-(x/\ell)^2)^{-1/2}dx
=
\ell\int_0^1(1-x^2)^{-1/2}dx
=\frac{\pi}{2}\ell.
\label{b3a}
\end{equation}
Next, we use the asymptotic formula for the Gamma function (see e.g. \cite[(6.1.47)]{AS})
$$
\frac{\Gamma(z+a)}{\Gamma(z+b)}=z^{a-b}(1+O(1/z)), \quad z\to+\infty.
$$
For the function $\gamma(x)$ from \eqref{b1} this yields
$$
\gamma(x)=\sqrt{2/x}(1+O(1/x)), \quad x\to\infty.
$$
Let $0\leq m\leq \ell-\ell^{1/3}$ and $\ell-m$ even; then
\begin{align*}
A_{\ell,m}^2
&=
\frac{2\ell+1}{\pi}\gamma(\ell-m)\gamma(\ell+m)
\\
&=\frac{2\ell}{\pi}(1+O(\ell^{-1}))
\sqrt{\frac{2}{\ell-m}}(1+O(\ell^{-1/3}))
\sqrt{\frac{2}{\ell+m}}(1+O(\ell^{-1}))
\\
&=\frac{4}{\pi}\frac{\ell}{\sqrt{(\ell-m)(\ell+m)}}(1+O(\ell^{-1/3}))
\\
&=\frac{4}{\pi}(1-(m/\ell)^2)^{-\frac12}(1+O(\ell^{-1/3})),
\end{align*}
and so
$$
\sum_{\genfrac{}{}{0pt}{1}{0\leq m\leq \ell-\ell^{1/3}}{\text{$m-\ell$ even}}}
\Abs{A_{\ell,m}^2-\frac{4}{\pi}(1-(m/\ell)^2)^{-\frac12}}
\leq
C\ell^{-1/3}\sum_{\genfrac{}{}{0pt}{1}{0\leq m\leq \ell-\ell^{1/3}}{\text{$m-\ell$ even}}}
(1-(m/\ell)^2)^{-1/2}.
$$
Using \eqref{b3} (on ignoring the condition $\ell-m$ even), we conclude that
$$
\sum_{\genfrac{}{}{0pt}{1}{0\leq m\leq \ell-\ell^{1/3}}{\text{$m-\ell$ even}}}
\Abs{A_{\ell,m}^2-\frac{4}{\pi}(1-(m/\ell)^2)^{-\frac12}}
=
O(\ell^{2/3}).
$$
Let us now estimate the sum over $\ell-\ell^{1/3}\leq m< \ell$.
We have
\begin{align*}
\sum_{\ell-\ell^{1/3}\leq m< \ell}(1-(m/\ell)^2)^{-1/2}
&\leq
\int_{\ell-\ell^{1/3}}^\ell(1-(x/\ell)^2)^{-1/2}dx
\\
&=
\ell\int_{1-\ell^{-2/3}}^1(1-x^2)^{-1/2}dx
=
\ell \ O(\ell^{-1/3})=O(\ell^{2/3})
\end{align*}
and similarly, using that $\gamma(x)=O(x^{-1/2})$,
\begin{align*}
\sum_
{\ell-\ell^{1/3}\leq m< \ell}
A_{\ell,m}^2
&\leq
C\ell \sum_{\ell-\ell^{1/3}\leq m< \ell}\gamma(\ell-m)\gamma(\ell+m)
\\
&\leq
C\ell^{1/2}\sum_{1\leq m\leq \ell^{1/3}}\gamma(m)
=
C\ell^{1/2}\sum_{1\leq m\leq \ell^{1/3}}O(m^{-1/2})
\\
&\leq
C\ell^{1/2}\ell^{1/6}
=
C\ell^{2/3}.
\end{align*}
This yields \eqref{b5}. The estimate \eqref{b4} follows from \eqref{b3} and \eqref{b5}.

Finally, let us check \eqref{b5a}; we have
$$
\max_{\abs{m}\leq \ell}
A_{\ell,m}^2 \le
\frac{2\ell+1}{\pi}\max_{\abs{m}\leq \ell}\gamma(\ell-m)\gamma(\ell+m)
=\frac{2\ell+1}{\pi}O(1/\sqrt{\ell})=O(\sqrt{\ell}),
$$
as required.
\end{proof}

\subsection{Proof of Lemma~\ref{lma.b1}}

\begin{proof}
Since $\sigma$ is bounded, we have
\begin{align*}
\norm{V_\ell[\sigma]}
&=
\norm{W_\ell[\sigma]}
=
\norm{C[x_\ell]M[\sigma]C[x_\ell]}
\leq
\norm{M[\sigma]}\norm{C[x_\ell]}^2
\\
&=
\norm{M[\sigma]}\norm{C[x_\ell]^2}.
\end{align*}
Using the notation $y_\ell$ (see \eqref{eq:y ell def}),
we find $C[x_\ell]^2=C[y_\ell]$.
Finally, from  \eqref{b0a}
$$
\norm{C[y_\ell]}=\sup_m A_{\ell,m}^2=O(\sqrt{\ell}),
$$
where we have used Lemma~\ref{lma.b2} at the last step. The proof is complete.
\end{proof}

\subsection{Asymptotics of traces of model operators}
Lemma~\ref{lma.b2} suggests that one may be able to replace $x_\ell$ in \eqref{b2a} by $z_\ell$, as defined in \eqref{b6}.
In preparation for this, we need a proposition.
Observe that the Fourier coefficients of $z_\ell$ have the form
$\omega(m/\ell)$ with $\omega(x)=2\pi^{-1/2}(1-x^2)_+^{-1/4}$.
The proposition below discusses the case where $\omega$ is replaced by a smooth function.

\begin{proposition}\label{prp.b4}
Let $\omega\in C_0^\infty(-1,1)$ be a non-negative function, and let $\omega_\ell$ be the function on $\partial\Omega$,
$$
\omega_\ell(\varphi)
=
\sum_{\genfrac{}{}{0pt}{1}{m=-\ell+2}{\text{$m-\ell$ \rm even}}}^{\ell-2}
\omega(m/\ell)e^{im\varphi}.
$$
Then for any test function $f\in C^\infty(\bbR)$ vanishing at the origin,
$$
\lim_{\ell\to\infty}
\frac{1}{\ell+1}\Tr f(C[\omega_\ell]M[\sigma_\text{\rm even}]C[\omega_\ell]^*)
=
\frac1{4\pi}\int_{-\pi}^{\pi}
\int_{-1}^1 f\bigl(\abs{\omega(\xi)}^2\sigma_\text{\rm even}(\varphi)\bigr)d\xi\,  d\varphi.
$$
\end{proposition}

The operator $C[\omega_\ell]M[\sigma_\text{\rm even}]C[\omega_\ell]^*$ is essentially a pseudodifferential operator with the symbol $\abs{\omega(\xi)}^2\sigma_\text{\rm even}(\varphi)$ and the semiclassical parameter $1/\ell$, and so the above proposition can be derived from the semiclassical spectral theory for general pseudodifferential operators, see e.g. \cite{DS}. In order to make the paper self-contained, we give a direct proof in the Appendix.

Our next aim is to replace a smooth function $\omega$ in the previous proposition by the function $2\pi^{-1/2}(1-x^2)_+^{-1/4}$ which has singularities at $x=\pm1$. This is done through an approximation argument.
We need to prepare a  general operator-theoretic statement.

\subsection{Aside on eigenvalue estimates}

\begin{proposition}\label{prp.b5}
Let $A$ and $B$ be finite rank self-adjoint operators in a Hilbert space  and let $f\in C^1(\bbR)$.
Then
$$
\Abs{\Tr(f(A)-f(B))}\leq \norm{f'}_{L^\infty}\norm{A-B}_{\Sch_1}.
$$
\end{proposition}
\begin{proof}
This estimate is well-known in the framework of the spectral shift function theory (see e.g. \cite{BYa}),
but for the sake of completeness we give a direct proof in this simple case.
In fact, we have already seen a similar argument in the proof of Theorem~\ref{thm.a2} in Section~\ref{sec.b6}.

Since $A$ and $B$ are finite rank, we may assume that our Hilbert space is finite dimensional, i.e. $A$ and $B$ are both $N\times N$ Hermitian matrices.
Denote $X=B-A$ and write $X=X_+-X_-$ with both $X_+$ and $X_-$ positive semi-definite. Then
$$
-X_-\leq X\leq X_+
$$
and therefore by the variational principle (with eigenvalues ordered non-increasingly) we have
\begin{align*}
\lambda_n(A-X_-)&\leq \lambda_n(A)\leq \lambda_n(A+X_+),
\\
\lambda_n(A-X_-)&\leq \lambda_n(B)\leq \lambda_n(A+X_+)
\end{align*}
for $n=1,\dots,N$.
It follows that for all $n$,
\[
\abs{f(\lambda_n(B))-f(\lambda_n(A))}
\leq
\norm{f'}_{L^\infty}(\lambda_n(A+X_+)-\lambda_n(A-X_-)).
\]
Summing over $n=1,\dots,N$, we find
\begin{align*}
\Abs{\Tr(f(A)-f(B))}
&\leq
\norm{f'}_{L^\infty}(\Tr(A+X_+)-\Tr(A-X_-))
\\
&=
\norm{f'}_{L^\infty}\Tr(X_++X_-).
\end{align*}
Finally, using the spectral decomposition of $X$, the operators $X_+$ and $X_-$ can be chosen such that $\norm{X}_{\Sch_1}=\Tr(X_++X_-)$.
\end{proof}

\subsection{Swapping $\omega_\ell$ for $z_\ell$}

\begin{theorem}
Let $z_\ell$ be as in \eqref{b6}.
Then for any compactly supported test function $f\in C^\infty(\bbR)$ vanishing at the origin,
$$
\lim_{\ell\to\infty}
\frac{1}{\ell+1}\Tr f(C[z_\ell]M[\sigma_\text{\rm even}]C[z_\ell])
=
\frac1{4\pi}\int_{-\pi}^{\pi}
\int_{-1}^1 f\left(\frac{4\sigma_\text{\rm even}(\varphi)}{\pi\sqrt{1-\xi^2}}\right)d\xi\,  d\varphi.
$$
\end{theorem}
\begin{proof}
\emph{Step 1: trace class estimate.}
For any $\eps>0$, let $\omega^{(\eps)}\in C_0^\infty(-1,1)$ be a non-negative function such that
\begin{align*}
\omega^{(\eps)}(\xi)& = 2\pi^{-1/2}(1-\xi^2)_+^{-1/4}, \quad \abs{\xi}<1-\eps,
\\
\omega^{(\eps)}(\xi)&\leq 2\pi^{-1/2}(1-\xi^2)_+^{-1/4}, \quad 1-\eps\leq\abs{\xi}<1.
\end{align*}
Using \eqref{b0b} and  \eqref{b3a}, we obtain:
\begin{align*}
\norm{C[z_\ell]}_{\Sch_2}^2
&=
\frac{4}{\pi}
\sum_{\genfrac{}{}{0pt}{1}{m=-\ell+2}{\text{$m-\ell$ even}}}^{\ell-2}
(1-(m/\ell)^2)^{-1/2}\leq 2\ell,
\\
\norm{C[\omega^{(\eps)}_\ell]}_{\Sch_2}^2
&=
\sum_{\genfrac{}{}{0pt}{1}{m=-\ell+2}{\text{$m-\ell$ even}}}^{\ell-2}
\abs{\omega^{(\eps)}(m/\ell)}^2
\leq 2\ell,
\\
\norm{C[z_\ell]-C[\omega^{(\eps)}_\ell]}_{\Sch_2}^2
&\leq
\frac{4}{\pi}
\sum_{\genfrac{}{}{0pt}{1}{1-\eps<\abs{m/\ell}<1}{\text{$m-\ell$ even}}}
(1-(m/\ell)^2)^{-1/2}
\leq C\ell\sqrt{\eps}.
\end{align*}
Writing
\begin{align*}
C[z_\ell]&M[\sigma_\text{\rm even}]C[z_\ell]
-
C[\omega^{(\eps)}_\ell]M[\sigma_\text{\rm even}]C[\omega^{(\eps)}_\ell]
\\
&=
C[z_\ell-\omega^{(\eps)}_\ell]M[\sigma_\text{\rm even}]C[z_\ell]
+
C[\omega^{(\eps)}_\ell]M[\sigma_\text{\rm even}]C[z_\ell-\omega_\ell^{(\eps)}]
\end{align*}
and using the triangle inequality for the trace norm and the Cauchy-Schwarz inequality in Schatten classes, we find
\begin{align*}
\norm{C[z_\ell]&M[\sigma_\text{\rm even}]C[z_\ell]
-
C[\omega^{(\eps)}_\ell]M[\sigma_\text{\rm even}]C[\omega^{(\eps)}_\ell]}_{\Sch_1}
\\
\leq&
\norm{C[z_\ell-\omega^{(\eps)}_\ell]}_{\Sch_2}
\norm{M[\sigma_\text{\rm even}]}
\norm{C[z_\ell]}_{\Sch_2}
\\
&+
\norm{C[\omega^{(\eps)}_\ell]}_{\Sch_2}
\norm{M[\sigma_\text{\rm even}]}
\norm{C[z_\ell-\omega^{(\eps)}]}_{\Sch_2}
\leq
C\ell\eps^{1/4}.
\end{align*}
From here using Proposition~\ref{prp.b5}, we find
\[
\Abs{
\Tr f(C[z_\ell]M[\sigma_\text{\rm even}]C[z_\ell])
-
\Tr f(C[\omega^{(\eps)}_\ell]M[\sigma_\text{\rm even}]C[\omega^{(\eps)}_\ell])
}
\leq
C\norm{f'}_{L^\infty}\ell\eps^{1/4}.
\]

\emph{Step 2: estimating $\limsup$ and $\liminf$.}
From the previous step, using  Proposition~\ref{prp.b4} with $\omega=\omega^{(\eps)}$, we find
\begin{align*}
\limsup_{\ell\to\infty}&
\frac{1}{\ell+1}\Tr f(C[z_\ell]M[\sigma_\text{\rm even}]C[z_\ell])
\\
&\leq
\limsup_{\ell\to\infty}
\frac{1}{\ell+1}\Tr f(C[\omega^{(\eps)}_\ell]M[\sigma_\text{\rm even}]C[\omega^{(\eps)}_\ell])
+
C\norm{f'}_{L^\infty}\eps^{1/4}
\\
&=
\frac1{4\pi}\int_{-\pi}^{\pi}
\int_{-1}^1 f\left(\sigma_\text{\rm even}(\varphi)
\abs{\omega^{(\eps)}(\xi)}^2
\right)d\xi\,  d\varphi
+
C\norm{f'}_{L^\infty}\eps^{1/4}
\end{align*}
and similarly
\begin{align*}
\liminf_{\ell\to\infty}&
\frac{1}{\ell+1}\Tr f(C[z_\ell]M[\sigma_\text{\rm even}]C[z_\ell])
\\
&\geq
\frac1{4\pi}\int_{-\pi}^{\pi}
\int_{-1}^1 f\left(\sigma_\text{\rm even}(\varphi)
\abs{\omega^{(\eps)}(\xi)}^2
\right)d\xi\,  d\varphi
-
C\norm{f'}_{L^\infty}\eps^{1/4}.
\end{align*}
Sending $\eps\to0$, we obtain the required statement.
\end{proof}

\subsection{Proof of Theorem~\ref{thm.b1}}\label{sec.3.10}
First let $f\in C^\infty(\bbR)$ be compactly supported so that $f(0)=0$.
By \eqref{b2}, it suffices to prove the relation
$$
\lim_{\ell\to\infty}
\frac{1}{\ell+1}\Tr f(C[x_\ell]M[\sigma]C[x_\ell])
=
\frac1{4\pi}\int_{-\pi}^{\pi}
\int_{-1}^1 f\left(\frac{4\sigma_\text{\rm even}(\varphi)}{\pi\sqrt{1-\xi^2}}\right)d\xi\,  d\varphi.
$$
First let us estimate the Hilbert-Schmidt norm of $C[x_\ell-z_\ell]$.
Using Lemma~\ref{lma.b2}, we find
\begin{align*}
\norm{C[x_\ell-z_\ell]}_{\Sch_2}^2
&=
\sum_{\genfrac{}{}{0pt}{1}{m=-\ell+2}{\text{$m-\ell$ even}}}^{\ell-2}
\abs{A_{\ell,m}-\frac{2}{\sqrt{\pi}}(1-(m/\ell)^2)^{-1/4}}^2
+2A_{\ell,\ell}^2
\\
&\leq
\sum_{\genfrac{}{}{0pt}{1}{m=-\ell+2}{\text{$m-\ell$ even}}}^{\ell-2}
\abs{A_{\ell,m}^2-\frac{4}{\pi}(1-(m/\ell)^2)^{-1/2}}
+2A_{\ell,\ell}^2
\\
&\leq
\sum_{\genfrac{}{}{0pt}{1}{m=-\ell+1}{\text{$m-\ell$ even}}}^{\ell-1}
\abs{A_{\ell,m}^2-\frac{4}{\pi}(1-(m/\ell)^2)^{-1/2}}
+2A_{\ell,\ell}^2
=O(\ell^{2/3}).
\end{align*}
From here, as in the previous proof, writing
\begin{align*}
C[x_\ell]M[\sigma]C[x_\ell]
-
C[z_\ell]M[\sigma]C[z_\ell]
\\
=
C[x_\ell-z_\ell]M[\sigma]C[x_\ell]
+
C[z_\ell]M[\sigma]C[x_\ell-z_\ell]
\end{align*}
and using the Cauchy-Schwarz for the Schatten norm, we find
$$
\norm{C[x_\ell]M[\sigma]C[x_\ell]
-
C[z_\ell]M[\sigma]C[z_\ell]}_{\Sch_1}
=O(\ell^{5/6}).
$$
From here and the previous theorem, using Proposition~\ref{prp.b5}, we obtain \eqref{eq.extra}.

\vspace{2mm}

Next let us consider the case $f(x)=x$. By the cyclicity of trace we find
\[
\Tr(C[x_\ell]M[\sigma]C[x_\ell])=\Tr(M[\sigma]C[x_\ell]^2)=\Tr(M[\sigma]C[y_\ell]).
\]
Now we can directly evaluate the trace by integrating the integral kernel of $M[\sigma]C[y_\ell]$ over the diagonal:
\[
\Tr(M[\sigma]C[y_\ell])
=
\frac1{2\pi}\int_{-\pi}^\pi\sigma(\varphi)y_\ell(0)d\varphi
=
\frac1{2\pi}\int_{-\pi}^\pi\sigma(\varphi)d\varphi
\sum_{\genfrac{}{}{0pt}{1}{m=-\ell}{\text{$m-\ell$ even}}}^\ell A_{\ell,m}^2.
\]
Using Lemma~\ref{lma.b2}, we find
\[
\sum_{\genfrac{}{}{0pt}{1}{m=-\ell}{\text{$m-\ell$ even}}}^\ell A_{\ell,m}^2
=
\frac{4}{\pi}
\sum_{\genfrac{}{}{0pt}{1}{m=-\ell}{\text{$m-\ell$ even}}}^\ell (1-(m/\ell)^2)^{-\frac12}
+O(\ell^{2/3})
=
\frac{2\ell}{\pi}\int_{-1}^1(1-\xi^2)^{-\frac12}d\xi+o(\ell)
\]
as $\ell\to\infty$. Putting this together, we obtain \eqref{eq.extra} for $f(x)=x$.
\qed


\section{The Birman-Schwinger principle and resolvent norm estimates}
\label{sec.cc}

\subsection{Definitions}
As in Section~\ref{sec.2.4}, $H[0]$ is the Neumann Laplacian; for typographical reasons, in this section we will write $H_0$ instead of $H[0]$.
As in Section~\ref{sec:Pl def}, we denote by ${\mathbf P}_\ell$ the orthogonal projection in $L^2(\Omega)$ onto the $(\ell+1)$-dimensional eigenspace of $H_0$ corresponding to the eigenvalue $\ell(\ell+1)$. We also denote by  ${\mathbf P}_\ell^\perp$ the projection onto the orthogonal complement to this eigenspace in $L^2(\Omega)$, so that ${\mathbf P}_\ell+{\mathbf P}_\ell^\perp=I$, the identity operator. Below we will make use of the orthogonal decomposition
$$
L^2(\Omega)=\Ran {\mathbf P}_\ell\oplus \Ran {\mathbf P}_\ell^\perp.
$$
Observe that the Neumann Laplacian $H_0$ is diagonal with respect to this decomposition.

For a fixed $\ell$, let $h_\ell^\perp[\sigma]$ be the restriction of the quadratic form $h[\sigma]$ onto the subspace $\Ran {\mathbf P}_\ell^\perp$, i.e.
\begin{equation}
h_\ell^\perp[\sigma](u)
=
\int_{\Omega}\abs{\nabla u}^2dx+\int_{-\pi}^\pi\sigma(\varphi)\abs{u(\pi/2,\varphi)}^2d\varphi,
\quad
u\in \Ran {\mathbf P}_\ell^\perp\cap W^{1}_{2}(\Omega).
\label{d1}
\end{equation}
We denote by $H_\ell^\perp[\sigma]$ the self-adjoint operator in $\Ran {\mathbf P}_\ell^\perp$, corresponding to this quadratic form.

Let
\begin{equation}
\label{eq:R(lambda) res def}
R(\lambda)=(H_0-\lambda I)^{-1}
\end{equation}
be the resolvent of the Neumann Laplacian. This resolvent is a meromorphic function of $\lambda$ with poles at $\lambda=\ell(\ell+1)$, $\ell=0,1,2,\dots$, and can be represented as the sum
\[
R(\lambda)=\sum_{\ell=0}^\infty \frac{{\mathbf P}_\ell}{\ell(\ell+1)-\lambda}
\]
convergent in the strong operator topology.
For a fixed $\ell$, we set
\begin{equation}
\label{eq:R res perp def}
R_\ell^\perp(\lambda)=\sum_{k\not=\ell}\frac{{\mathbf P}_k}{k(k+1)-\lambda}=R(\lambda)-\frac{{\mathbf P}_\ell}{\ell(\ell+1)-\lambda},
\end{equation}
i.e. we subtract from $R(\lambda)$ its singular part at $\lambda=\ell(\ell+1)$.

\subsection{Restrictions onto the boundary}
For any real $\lambda$ that is not an eigenvalue of $H_0$, we define $\abs{H_0-\lambda}^{-1/2}$ using the standard functional calculus for self-adjoint operators, and set
\[
(H_0-\lambda)^{-1/2}:=\abs{H_0-\lambda}^{-1/2}\sign(H_0-\lambda);
\]
then
$$
(H_0-\lambda)^{-1}=(H_0-\lambda)^{-1/2}\abs{H_0-\lambda}^{-1/2}.
$$
Observe that $\abs{H_0-\lambda}^{-1/2}$ and $(H_0-\lambda)^{-1/2}$ map $L^2(\Omega)$ to the Sobolev class $W^1_2(\Omega)$. Next, let $T: W^1_2(\Omega)\to L^2(\partial\Omega)$ be the operator of restriction onto the boundary (also known as the trace operator). This operator is bounded and compact, see e.g. \cite[Theorem 2.6.2]{Necas}. It follows that the operators
$$
T\abs{H_0-\lambda}^{-1/2}, \quad T(H_0-\lambda)^{-1/2}: L^2(\Omega)\to L^2(\partial\Omega)
$$
are bounded and compact.

Below we use the restrictions of the operators $R(\lambda)$ and $R_\ell^\perp(\lambda)$ in \eqref{eq:R(lambda) res def} and \eqref{eq:R res perp def}
respectively onto $L^2(\partial\Omega)$. We denote these restrictions by $\wt R(\lambda)$ and $\wt R_\ell^\perp(\lambda)$ respectively. Heuristically, these operators are defined by restricting the integral kernels onto the boundary. In order to define them rigorously, we use the above operator $T$ as follows:
\begin{align*}
\wt R(\lambda)&=T\abs{H_0-\lambda}^{-1/2}(T(H_0-\lambda)^{-1/2})^*,
\\
\wt R_\ell^\perp(\lambda)&=T\abs{H_0-\lambda}^{-1/2}{\mathbf P}_\ell^\perp(T(H_0-\lambda)^{-1/2})^*.
\end{align*}
Finally, we will also use the operator $\wt{\mathbf P}_\ell$ in $L^2(\partial\Omega)$, which is obtained by restricting the integral kernel of ${\mathbf P}_\ell$ onto the equator. In this case, the integral kernel is infinitely smooth and the definition is straightforward. With these definitions, we have
\begin{equation}
\label{eq:R perp tilde sum}
\wt R(\lambda)=\sum_{k=0}^\infty\frac{\wt{\mathbf P}_k}{k(k+1)-\lambda},
\quad
\wt R_\ell^\perp(\lambda)=\sum_{k\not=\ell}\frac{\wt{\mathbf P}_k}{k(k+1)-\lambda}.
\end{equation}

\subsection{The Birman-Schwinger principle}
Below $\sigma_0$ is a non-zero constant.
\begin{proposition}
Let $\lambda\in\bbR\setminus\{\ell(\ell+1): \ell\geq0\}$. Then
\begin{equation}
\dim\Ker(H[\sigma_0]-\lambda I)
=
\dim\Ker(I+\sigma_0\wt R(\lambda)).
\label{c2a}
\end{equation}
Similarly,
\begin{equation}
\dim\Ker(H_\ell^\perp[\sigma_0]-\lambda I)
=
\dim\Ker(I+\sigma_0\wt R_\ell^\perp(\lambda)).
\label{c2}
\end{equation}
\end{proposition}
\begin{proof}
Let us prove \eqref{c2a}.
As both $\abs{H_0-\lambda}^{-1/2}$ and $(H_0-\lambda)^{-1/2}$ are linear isomorphisms,
we have
\begin{align*}
\dim\Ker(H[\sigma_0]-\lambda I)
&=
\dim\Ker\biggl(\abs{H_0-\lambda}^{-1/2}(H[\sigma_0]-\lambda I)(H_0-\lambda)^{-1/2}\biggr).
\end{align*}
We recall that the quadratic form of $H[\sigma_0]$ is
\begin{align*}
h[\sigma_0](u)&=
\int_{\Omega}\abs{\nabla u}^2dx+\sigma_0\int_{-\pi}^\pi\abs{u(\pi/2,\varphi)}^2d\varphi
\\
&=h[0](u)+\sigma_0\int_{-\pi}^\pi\abs{u(\pi/2,\varphi)}^2d\varphi,
\quad u\in W^1_2(\Omega).
\end{align*}
From here we obtain
$$
\abs{H_0-\lambda}^{-1/2}(H[\sigma_0]-\lambda I)(H_0-\lambda)^{-1/2}
=
I+\sigma_0(T\abs{H_0-\lambda}^{-1/2})^*(T(H_0-\lambda)^{-1/2}).
$$
Using the identity
$$
\dim\Ker(I+K_1K_2)=\dim\Ker(I+K_2K_1)
$$
for any two compact operators $K_1$, $K_2$, we find
\begin{align*}
\dim\Ker(H[\sigma_0]-\lambda I)
&=
\dim\Ker\biggl(I+\sigma_0 (T\abs{H_0-\lambda}^{-1/2})^*T(H_0-\lambda)^{-1/2}\biggr)
\\
&=
\dim\Ker\biggl(I+\sigma_0 T(H_0-\lambda)^{-1/2}(T\abs{H_0-\lambda}^{-1/2})^*\biggr)
\\
&=\dim\Ker(I+\sigma_0\wt R(\lambda)),
\end{align*}
which proves \eqref{c2a}.
The proof of \eqref{c2} follows along the same steps, with the only change that instead of $H_0$ one needs to consider the restriction of $H_0$ onto the range of ${\mathbf P}_\ell^\perp$.

\end{proof}

\subsection{Norm estimate for $\wt R_\ell^\perp(\lambda)$}

\begin{lemma}\label{lma.cc1}
We have
$$
\sup_{\lambda\in[\ell^2,(\ell+1)^2]}\norm{\wt R_\ell^\perp(\lambda)}=O(\ell^{-1/2}\log \ell),\quad \ell\to\infty.
$$
\end{lemma}
\begin{proof}
We use the second identity in \eqref{eq:R perp tilde sum}.
Recall that the trigonometric polynomial $y_{\ell}$ is defined in \eqref{eq:y ell def}.
Then, by the definition \eqref{b11a} of $\mathbf{P}_{\ell}$ as an integral operator, and
by the definition of $\widetilde{\mathbf{P}}_{\ell}$ as its restriction onto the boundary, we see that $\widetilde{\mathbf{P}}_{\ell}$ is the convolution operator
$$\widetilde{\mathbf{P}}_{\ell} = C[y_{\ell}],$$ as in \eqref{eq:C[a] conv op def}.
Hence, by the linearity, the operator $\widetilde{R}_{\ell}^{\perp}(\lambda)$ in \eqref{eq:R perp tilde sum} is the convolution operator
$$\widetilde{R}_{\ell}^{\perp}(\lambda)=C[F],$$ on the unit circle, with the function
$$F(\varphi)=F_{\ell;\lambda}(\varphi) = \sum\limits_{k\ne \ell}\frac{1}{k(k+1)-\lambda}y_{k}(\varphi).$$

The Fourier coefficients of $F$ are the numbers
\begin{equation*}
a_{m}=a_{\ell;m}:= \sum\limits_{k\ne \ell}\frac{1}{k(k+1)-\lambda}A_{k,m}^{2},
\end{equation*}
with $A_{\ell,m}$ as in \eqref{eq:restr equator Alm2} (treated in Lemma \ref{lma.b2}); we recall that by our convention, $A_{\ell,m}=0$ unless $\ell-m$ is even and $\abs{m}\leq \ell$. 
Since the operator norm of $C[F]$ is the supremum of the Fourier coefficients of $F$, we have
\begin{equation*}
\norm{\wt R_\ell^\perp(\lambda)} = \sup\limits_{m\in \Z}|a_{m}| =
\sup\limits_{m\in\Z}\left| \sum\limits_{k\ne \ell}\frac{1}{k(k+1)-\lambda}A_{k,m}^{2} \right|,
\end{equation*}
that is to be bounded, uniformly w.r.t. $\lambda\in[\ell^2,(\ell+1)^2]$.

To bound the numbers $a_{m}$ we employ a strategy similar to the proof of Lemma \ref{lma.b1}.
At first we use the uniform upper bound \eqref{b5a} of Lemma \ref{lma.b2}, together with the triangle inequality, to yield
\[
|a_{m}| \le C\sum\limits_{k\ne \ell}\frac{\sqrt{k}}{|k(k+1)-\lambda|}.
\]
Now the result follows from the next lemma.
\end{proof}

\begin{lemma}
We have
$$
\sup_{\lambda\in[\ell^2,(\ell+1)^2]}\sum\limits_{k\ne \ell}\frac{\sqrt{k}}{|k(k+1)-\lambda|}
=
O(\ell^{-1/2}\log \ell),\quad \ell\to\infty.
$$
\end{lemma}
\begin{proof}
We consider separately the sums over $k\leq\ell-1$ and over $k\geq\ell+1$. For the first sum we have
$$
\sum\limits_{k=1}^{\ell-1}\frac{\sqrt{k}}{|k(k+1)-\lambda|}
\leq
\sum\limits_{k=1}^{\ell-1}\frac{\sqrt{k}}{\ell^2-k(k+1)}.
$$
Clearly, every term in the series here is $O(\ell^{-1/2})$. Thus, we can write
\begin{align*}
\sum\limits_{k=1}^{\ell-1}\frac{\sqrt{k}}{\ell^2-k(k+1)}
&=
O(\ell^{-1/2})+\sum\limits_{k=1}^{\ell-3}\frac{\sqrt{k}}{\ell^2-k(k+1)}
\\
&\leq
O(\ell^{-1/2})+\int_1^{\ell-2}\frac{\sqrt{x}}{\ell^2-x(x+1)}dx.
\end{align*}
For the integral in the right hand side we find
\begin{align*}
\int_1^{\ell-2}\frac{\sqrt{x}}{\ell^2-x(x+1)}dx
&\leq
\int_1^{\ell-2}\frac{\sqrt{x+1}}{\ell^2-(x+1)^2}dx
\\
&=
\int_2^{\ell-1}\frac{\sqrt{x}}{\ell^2-x^2}dx
\leq
\int_0^{\ell-1}\frac{\sqrt{x}}{\ell^2-x^2}dx.
\end{align*}
By the change of variable $x=\ell y$ we finally obtain
$$
\int_0^{\ell-1}\frac{\sqrt{x}}{\ell^2-x^2}dx
=
\ell^{-1/2}\int_0^{1-1/\ell}\frac{\sqrt{y}}{1-y^2}dy
=
O(\ell^{-1/2}\log \ell)
$$
as $\ell\to\infty$.

Similarly, for the second sum we find
$$
\sum\limits_{k=\ell+1}^{\infty}\frac{\sqrt{k}}{|k(k+1)-\lambda|}
\leq
\sum\limits_{k=\ell+1}^{\infty}\frac{\sqrt{k}}{k(k+1)-(\ell+1)^{2}}.
$$
For the denominator, we have
$$
k(k+1)-(\ell+1)^{2}=(k-\ell)k + (\ell+1)(k-(\ell+1))\ge (k-\ell)k,
$$
and so
\begin{align*}
\sum\limits_{k=\ell+1}^{\infty}\frac{\sqrt{k}}{k(k+1)-(\ell+1)^2}
&\leq
\sum\limits_{k=\ell+1}^{\infty}\frac{1}{\sqrt{k}(k-\ell)}
=
\frac1{\sqrt{\ell+1}}
+
\sum\limits_{k=\ell+2}^{\infty}\frac{1}{\sqrt{k}(k-\ell)}
\\
&\leq\frac1{\sqrt{\ell+1}}
+
\int_{\ell+1}^\infty \frac{dx}{\sqrt{x}(x-\ell)}
\\
&\leq\frac1{\sqrt{\ell+1}}
+\ell^{-1/2}
\int_{1+1/\ell}^{\infty}\frac{dy}{\sqrt{y}(y-1)}
=
O(\ell^{-1/2}\log\ell),
\end{align*}
as required.
\end{proof}

\section{Proof of Theorem~\ref{thm.d2}}
\label{sec.d}

\subsection{Eigenvalues of $H_\ell^\perp[\sigma]$}
We recall that the operator $H_\ell^\perp[\sigma]$ was defined in the beginning of the previous section.

\begin{lemma}\label{lma.d3}
Let $\sigma_0$ be a positive number.
Then for all sufficiently large $\ell$ and all continuous functions $\sigma$ satisfying $-\sigma_0\leq \sigma \leq \sigma_0$, the operators $H_\ell^\perp[\sigma]$ have no eigenvalues in the interval $\lambda\in[\ell^2,(\ell+1)^2]$. Furthermore, the number of eigenvalues of each operator $H_\ell^\perp[\sigma]$ in the interval $(-\infty,\ell(\ell+1))$ is $L=\ell(\ell+1)/2$.
\end{lemma}
\begin{proof}
First consider the case of constant $\sigma$; let $\sigma=t\sigma_0$, $t\in[-1,1]$.
We use the Birman-Schwinger principle \eqref{c2}. By Lemma~\ref{lma.cc1} for any $\lambda\in [\ell^2,(\ell+1)^2]$ we have
$$
\norm{\sigma_0 \wt R_\ell^\perp(\lambda)}<1
$$
for all sufficiently large $\ell$. It follows that the kernel of $I+t\sigma_0\wt R_\ell^\perp(\lambda)$ is trivial, and therefore $H_\ell^\perp[t\sigma_0]$ has no eigenvalues in $[\ell^{2},(\ell+1)^{2}]$ for all sufficiently large $\ell$.

Let us fix such sufficiently large $\ell$. The eigenvalues of $H_\ell^\perp[t\sigma_0]$ depend continuously on $t$ and do not enter the interval $[\ell^2,(\ell+1)^2]$ as $t$ varies from $-1$ to $1$. This means that the number of these eigenvalues below $\ell^{2}$ is independent on $t$. For $t=0$ these are the Neumann eigenvalues and we can count them explicitly: each Neumann eigenvalue $k(k+1)$, $k<\ell$, contributes multiplicity $k+1$; altogether we have $\ell(\ell+1)/2$ eigenvalues below $\ell^{2}$.

Now consider the case of a variable $\sigma$. By the assumption $-\sigma_0\leq \sigma\leq \sigma_0$, we find (see definition \eqref{d1})
$$
H_\ell^\perp[-\sigma_0]\leq H_\ell^\perp[\sigma]\leq H_\ell^\perp[\sigma_0].
$$
Denoting the eigenvalues of $H_\ell^\perp[\sigma]$ by $\lambda_n^\perp(\sigma)$, by the variational principle (see Section~\ref{sec.var}) we find
$$
\lambda_n^\perp(-\sigma_0)\leq \lambda_n^\perp(\sigma)\leq \lambda_n^\perp(\sigma_0)
$$
for all $n$. By the first part of the proof, we have, with $L=\ell(\ell+1)/2$,
$$
\lambda_L^\perp(\sigma_0)<\ell^{2} \quad\text{ and }\quad
\lambda_{L+1}^\perp(-\sigma_0)>(\ell+1)^{2}.
$$
From here we find
$$
\lambda_L^\perp(\sigma)<\ell^{2},
\quad\text{ and }\quad
\lambda_{L+1}^\perp(\sigma)>(\ell+1)^{2},
$$
as required.
\end{proof}

\subsection{A variational lemma}
It is convenient to express the first step of the proof of Theorem~\ref{thm.d2} as a separate lemma.
Recall that we denote $\sigma_\pm=\sigma\pm\eps\abs{\sigma}$.
\begin{lemma}\label{lma.var}
For any $\ell>0$ and any $\eps>0$, we have the inequalities
$$
H_\ell^-[\sigma]\leq H[\sigma]\leq H_\ell^+[\sigma],
$$
where the operators $H_\ell^-[\sigma]$ and $H_\ell^+[\sigma]$ are defined by the sums
\begin{align*}
H_\ell^+[\sigma]&=\biggl(\ell(\ell+1)I+V_\ell[\sigma_+]\biggr)\oplus H_\ell^\perp[\sigma+\tfrac1\eps\abs{\sigma}],
\\
H_\ell^-[\sigma]&=\biggl(\ell(\ell+1)I+V_\ell[\sigma_-]\biggr)\oplus H_\ell^\perp[\sigma-\tfrac1\eps\abs{\sigma}]
\end{align*}
in the orthogonal decomposition
$$
L^2(\Omega)=\Ran {\mathbf P}_\ell\oplus\Ran {\mathbf P}_\ell^\perp.
$$
\end{lemma}
\begin{proof}

For a fixed $\ell$ and for $u\in W^1_2(\Omega)$, we write
$$
u=u_\ell+u_\ell^\perp, \quad u_\ell={\mathbf P}_\ell u, \quad u_\ell^\perp={\mathbf P}_\ell^\perp u.
$$
Let us expand the quadratic form
$$
h[\sigma](u_\ell+u_\ell^\perp)=h[0](u_\ell+u_\ell^\perp)+\int_{-\pi}^\pi\sigma(\varphi)\abs{u_\ell(\pi/2,\varphi)+u_\ell^\perp(\pi/2,\varphi)}^2d\varphi.
$$
We observe that the Neumann Laplacian is diagonalised by the decomposition $u=u_\ell+u_\ell^\perp$, i.e.
$$
h[0](u_\ell+u_\ell^\perp)=h[0](u_\ell)+h[0](u_\ell^\perp).
$$
Since $u_\ell$ is in the eigenspace $\Ran {\mathbf P}_\ell$ of $H[0]$, corresponding to the eigenvalue $\ell(\ell+1)$, we have
$$
h[0](u_\ell)=\ell(\ell+1)\norm{u_\ell}^2_{L^2(\Omega)}.
$$
Using our notation \eqref{d1}, we can write
$$
h[0](u_\ell^\perp)
=
h_\ell^\perp[0](u_\ell^\perp).
$$
Thus,
$$
h[0](u_\ell+u_\ell^\perp)
=
\ell(\ell+1)\norm{u_\ell}^2
+
h_\ell^\perp[0](u_\ell^\perp).
$$

Now consider the integral over the equator:
\begin{align*}
\int_{-\pi}^\pi\sigma(\varphi)\abs{u_\ell(\pi/2,\varphi)+&u_\ell^\perp(\pi/2,\varphi)}^2d\varphi
\\
=&
\int_{-\pi}^\pi\sigma(\varphi)\abs{u_\ell(\pi/2,\varphi)}^2d\varphi
+
\int_{-\pi}^\pi\sigma(\varphi)\abs{u_\ell^\perp(\pi/2,\varphi)}^2d\varphi
\\
&+
2\Re \int_{-\pi}^\pi\sigma(\varphi)u_\ell(\pi/2,\varphi)\overline{u_\ell^\perp(\pi/2,\varphi)}  d\varphi.
\end{align*}
We use the estimate
$$
2\abs{ab}\leq \eps\abs{a}^2+\frac1{\eps}\abs{b}^2
$$
for the cross term in the last expression (with $a=u_\ell$ and $b=u_\ell^\perp$); this yields
\begin{align*}
\int_{-\pi}^\pi&\sigma(\varphi)\abs{u(\pi/2,\varphi)}^2d\varphi
\\
&\leq
\int_{-\pi}^\pi(\sigma(\varphi)+\eps\abs{\sigma(\varphi)})\abs{u_\ell(\pi/2,\varphi)}^2d\varphi
+
\int_{-\pi}^\pi(\sigma(\varphi)+\frac1\eps\abs{\sigma(\varphi)})\abs{u_\ell^\perp(\pi/2,\varphi)}^2d\varphi
\end{align*}
and similarly
\begin{align*}
\int_{-\pi}^\pi&\sigma(\varphi)\abs{u(\pi/2,\varphi)}^2d\varphi
\\
&\geq
\int_{-\pi}^\pi(\sigma(\varphi)-\eps\abs{\sigma(\varphi)})\abs{u_\ell(\pi/2,\varphi)}^2d\varphi
+
\int_{-\pi}^\pi(\sigma(\varphi)-\frac1\eps\abs{\sigma(\varphi)})\abs{u_\ell^\perp(\pi/2,\varphi)}^2d\varphi.
\end{align*}
Putting this all together and recalling our notation $v_\ell[\sigma]$ in \eqref{b8}, we find
\begin{align}
h[\sigma](u)&\leq
\ell(\ell+1)\norm{u_\ell}^2_{L^2(\Omega)}
+v_\ell[\sigma_+](u_\ell)+h_\ell^\perp[\sigma+\tfrac1\eps\abs{\sigma}](u_\ell^\perp),
\label{d2}
\\
h[\sigma](u)&\geq
\ell(\ell+1)\norm{u_\ell}^2_{L^2(\Omega)}
+v_\ell[\sigma_-](u_\ell)+h_\ell^\perp[\sigma-\tfrac1\eps\abs{\sigma}](u_\ell^\perp).
\label{d3}
\end{align}
Here \eqref{d2} can be rewritten as $H[\sigma]\leq H_\ell^+[\sigma]$, and \eqref{d3} as $H_\ell^-[\sigma]\leq H[\sigma]$.

We note that in this case the domains of the quadratic forms corresponding to all three operators ($H_\ell^-[\sigma]$, $H[\sigma]$ and $H_\ell^+[\sigma]$) coincide.
The proof of the lemma is complete.
\end{proof}

\subsection{Proof of Theorem~\ref{thm.d2}}
Combining Lemma~\ref{lma.var} with the variational principle (see Section~\ref{sec.var}), we find
\begin{equation}
\lambda_n(H_\ell^-[\sigma])\leq \lambda_n(H[\sigma])\leq \lambda_n(H_\ell^+[\sigma])
\label{d4}
\end{equation}
for all indices $n$.
Consider the upper bound here; recall that
$$
H_\ell^+[\sigma]=\biggl(\ell(\ell+1)I+V_\ell[\sigma_+]\biggr)\oplus H_\ell^\perp[\sigma+\tfrac1\eps\abs{\sigma}].
$$
The sequence of the eigenvalues of $H_\ell^+[\sigma]$ is the re-ordered union of the sequences of eigenvalues of the two components in this orthogonal sum.
By  Lemma~\ref{lma.d3}, the second component $H_\ell^\perp[\sigma+\tfrac1\eps\abs{\sigma}]$ does not contribute any eigenvalues to the interval $[\ell^2,(\ell+1)^2]$ and has $L=\ell(\ell+1)/2$ eigenvalues below this interval.
The first component has exactly $\ell+1$ eigenvalues given by
$$
\ell(\ell+1)+\lambda_k(V_\ell[\sigma_+]), \quad k=1,\dots, \ell+1.
$$
By Lemma~\ref{lma.b1}, all these eigenvalues are in the interval $\Lambda_\ell$.

It follows that for the eigenvalues of $H_\ell^+[\sigma]$ in the interval $[\ell^2,(\ell+1)^2]$ we have
$$
\lambda_{L+k}(H_\ell^+[\sigma])
=
\ell(\ell+1)+\lambda_k(V_\ell[\sigma_+]), \quad k=1,\dots,\ell+1.
$$
Combining this with the upper bound in \eqref{d4}, we obtain the upper bound in the statement of the theorem. The lower bound is obtained in the same way.
\qed

\subsection{Proof of Theorem~\ref{thm:a1}}
\label{sec:proof Lemma a1}
By Theorem~\ref{thm.d2} and Lemma~\ref{lma.b1}, the eigenvalues $\lambda_{L+k}(H[\sigma])$, $k=1,\dots,\ell+1$, belong to the interval $\Lambda_\ell$. This yields the required statement.
\qed

\section{Odd trigonometric polynomials: proof of Theorem~\ref{thm.a3}}
\label{sec:odd sigma}

\subsection{Notation and setup}

We denote by $\calH_\ell\subset L^2(\Omega)$ the linear space of spherical harmonics of degree $\ell$, i.e.
$$
\calH_\ell=\Span\{Y_\ell^m: -\ell\leq m\leq \ell\}.
$$
We decompose $\calH_\ell$ according to the parity of $m-\ell$:
$$
\calH_\ell=\calH_\ell^N\oplus\calH_\ell^D,
$$
where
\begin{align*}
\calH_\ell^N
&=
\Span\{Y_\ell^m: -\ell\leq m\leq \ell,\quad m-\ell\text{ even}\},
\\
\calH_\ell^D
&=
\Span\{Y_\ell^m: -\ell\leq m\leq \ell,\quad m-\ell\text{ odd}\}.
\end{align*}
Equivalently, $\calH_\ell^N$ (resp. $\calH_\ell^D$) is the subspace of functions satisfying the Neumann (resp. Dirichlet) boundary condition on the equator.

Further, we consider the corresponding subspaces of $L^2(\partial\Omega)$:
\begin{align*}
\wt\calH_\ell^N
&=
\Span\{e^{im\varphi}: -\ell\leq m\leq \ell,\quad m-\ell\text{ even}\},
\\
\wt\calH_\ell^D
&=
\Span\{e^{im\varphi}: -\ell\leq m\leq \ell,\quad m-\ell\text{ odd}\}.
\end{align*}

\begin{lemma}\label{lma.f1}
\begin{enumerate}[\rm (i)]
\item
For any $\ell$, the restriction map $F\mapsto F|_{\partial\Omega}$ is a linear isomorphism between $\calH_\ell^N$ and $\wt\calH_\ell^N$.
\item
For any $\ell$, the map $F\mapsto \frac{\partial F}{\partial n}|_{\partial\Omega}$ is a linear isomorphism between $\calH_\ell^D$ and $\wt\calH_\ell^D$.
\end{enumerate}
\end{lemma}
\begin{proof}
(i) We have
\begin{align*}
\{Y_\ell^m: \quad -\ell\leq m\leq \ell,\quad m-\ell\text{ even}\}\quad \text{is a linear basis in $\calH_\ell^N$},
\\
\{e^{im\varphi}:\quad -\ell\leq m\leq \ell,\quad m-\ell\text{ even}\}\quad \text{is a linear basis in $\wt\calH_\ell^N$}.
\end{align*}
By \eqref{b11} and \eqref{eq:restr equator Alm2}, we have
$$
Y_\ell^m|_{\partial\Omega}=Y_\ell^m(\pi/2,\varphi)=(-1)^{\frac{\ell+m}{2}}
\frac{A_{\ell,m}}{\sqrt{2\pi}}e^{im\varphi},
\quad
A_{\ell,m}\not=0,
$$
and so the restriction map relates these bases through multiplication by non-zero constants.
This completes the proof of (i).

(ii) Similarly,
\begin{align*}
\{Y_\ell^m: \quad -\ell\leq m\leq \ell,\quad m-\ell\text{ odd}\}\quad \text{is a linear basis in $\calH_\ell^D$},
\\
\{e^{im\varphi}:\quad -\ell\leq m\leq \ell,\quad m-\ell\text{ odd}\}\quad \text{is a linear basis in $\wt\calH_\ell^D$}.
\end{align*}
We have
$$
\frac{\partial Y_{\ell}^{m}}{\partial n}\bigg|_{\partial\Omega} = \frac{\partial Y_{\ell}^{m}}{\partial \theta} (\pi/2,\varphi) = B_{\ell,m}\cdot e^{im \varphi},
$$
with
$$
B_{\ell,m}:=-
\sqrt{\frac{(2\ell+1)}{2\pi}\frac{(\ell-m)!}{(\ell+m)!}}\frac{d}{dx}{P_{\ell}^{m}}(x)|_{x=0},
$$
that, we claim, does not vanish. Indeed, upon using the assumption that $\ell-m$ is odd, we explicitly have  ~\cite[p. 128]{RW} that
\begin{equation*}
\frac{d}{dx}{P_{\ell}^{m}}(x)|_{x=0} = \frac{2^{m+1}}{\sqrt{\pi}}\sin\left(\frac{\pi(m+\ell)}{2}\right)\frac{\Gamma\left( \frac{\ell+m}{2}+1 \right)}{\Gamma\left( \frac{\ell-m+1}{2} \right)} =
\pm\frac{2^{m+1}}{\sqrt{\pi}}\frac{\Gamma\left( \frac{\ell+m}{2}+1 \right)}{\Gamma\left( \frac{\ell-m+1}{2} \right)} \ne 0.
\end{equation*}
This completes the proof of (ii).
\end{proof}

\subsection{Proof of Theorem~\ref{thm.a3}}
Let $\sigma$ be as in the statement of the theorem, see \eqref{eq:def-odd-sigma}.
Recall that we denoted by $M[\sigma]$ the operator of multiplication by $\sigma$ in $L^2(\Omega)$.
Let us consider the action of  $M[\sigma]$ on functions from the space $\wt\calH_\ell^N$. Since $\sigma$ is odd, multiplication by $\sigma$ reverses the parity of functions on the boundary. It follows that $M[\sigma]$ maps a suitable subspace of $\wt\calH_\ell^N$ to $\wt\calH_\ell^D$. By ``suitable subspace'' here we mean restricting the degree of $f\in \wt\calH_\ell^N$ so that the product $\sigma f$ has degree $\leq \ell$. More precisely, we see that (using that $d$ is odd)
\begin{equation}
M[\sigma]: \wt\calH_{\ell-d-1}^N\to \wt\calH_\ell^D.
\label{f2}
\end{equation}

Motivated by this, let us consider the subspace $\calH_\ell^{N,-}\subset\calH_\ell^N$ defined by the condition
$$
\calH_\ell^{N,-}=\{F\in \calH_\ell^N: F|_{\partial\Omega}\in\wt\calH_{\ell-d-1}^N\}.
$$
Take any $F_N\in \calH_\ell^{N,-}$ and consider its restriction $f_N=F_N|_{\partial\Omega}$.
Then by \eqref{f2} we have $\sigma f_N\in \wt\calH_\ell^D$. By Lemma~\ref{lma.f1}(ii), there exists $F_D\in \calH_\ell^D$ such that
$$
\frac{\partial F_D}{\partial n}\bigg|_{\partial\Omega}=-\sigma f_N.
$$
Consider the function $F:=F_N+F_D\in\calH_\ell$.
Since $F_N$ (resp. $F_D$) satisfies the Neumann (resp. Dirichlet) boundary condition on the equator, we find
\begin{equation*}
\sigma \cdot F|_{\partial\Omega}+\frac{\partial F}{\partial n}\bigg|_{\partial\Omega}
=
\sigma \cdot F_{N}|_{\partial\Omega}+\frac{\partial F_D}{\partial n}\bigg|_{\partial\Omega}
=\sigma f_N-\sigma f_N=0
\end{equation*}
by construction. Thus, $F$ is a Robin eigenfunction with the eigenvalue $\ell(\ell+1)$.

It remains to estimate from below the dimension of the corresponding eigenspace. By inspection, the codimension of $\wt\calH_{\ell-d-1}^N$ in $\wt\calH_\ell^N$ is $d+1$. By Lemma~\ref{lma.f1}(i), the codimension of $\calH_\ell^{N,-}$ in $\calH_\ell^{N}$ is also $d+1$. Since $\dim \calH_\ell^{N}=\ell+1$, we find that $\dim\calH_\ell^{N,-}=\ell-d$. Thus, the multiplicity of the Robin eigenvalue $\lambda=\ell(\ell+1)$ is at least $\ell-d$.  It follows that all but $d+1$ eigenvalues in the $\ell$'th cluster (for large $\ell$) coincide with $\ell(\ell+1)$.
The proof of Theorem~\ref{thm.a3} is complete.

\appendix

\section{Proof of Proposition~\ref{prp.b4}}

Below $\omega$ and $\sigma$ are as in the hypothesis of Proposition~\ref{prp.b4}.
In order to make the formulas below more readable, we assume without loss of generality that $\sigma$ is even, $\sigma(\varphi+\pi)=\sigma(\varphi)$, so that $\sigma_\text{\rm even}=\sigma$.

\subsection{A commutator estimate}
\begin{lemma}
We have the Hilbert-Schmidt norm estimate
\begin{equation}
\norm{M[\sigma]C[\omega_\ell]-C[\omega_\ell]M[\sigma]}_{\Sch_2}=O(1), \quad \ell\to\infty.
\label{A1}
\end{equation}
\end{lemma}
\begin{proof}
The commutator in \eqref{A1} is the integral operator in $L^2(\partial\Omega)$ with the integral kernel
$$
\omega_\ell(\varphi-\varphi')(\sigma(\varphi)-\sigma(\varphi')).
$$
Since $\sigma$ is smooth and \emph{even}, we can write
$$
\abs{\sigma(\varphi)-\sigma(\varphi')}
\leq
C\abs{e^{2i(\varphi-\varphi')}-1}.
$$
It follows that the Hilbert-Schmidt norm can be estimated as
\begin{align*}
\norm{M[\sigma]C[\omega_\ell]&-C[\omega_\ell]M[\sigma]}_{\Sch_2}^2
\\
&=
\int_{-\pi}^\pi \int_{-\pi}^\pi
\abs{\omega_\ell(\varphi-\varphi')}^2\abs{\sigma(\varphi)-\sigma(\varphi')}^2 d\varphi\ d\varphi'
\\
&\leq
C\int_{-\pi}^\pi \int_{-\pi}^\pi
\abs{\omega_\ell(\varphi-\varphi')}^2
\abs{e^{2i(\varphi-\varphi')}-1}^2
d\varphi\ d\varphi'
\\
&=
C\int_{-\pi}^\pi
\abs{\omega_\ell(\varphi)}^2
\abs{e^{2i\varphi}-1}^2d\varphi.
\end{align*}
By Plancherel's theorem, we have
\begin{align*}
\int_{-\pi}^\pi
\abs{\omega_\ell(\varphi)}^2
\abs{e^{2i\varphi}-1}^2d\varphi
&=
\int_{-\pi}^\pi
\abs{\omega_\ell(\varphi)e^{2i\varphi}-\omega_\ell(\varphi)}^2d\varphi
\\
&=
\sum_{\genfrac{}{}{0pt}{1}{m=-\infty}{\text{$m-\ell$ \rm even}}}^{\infty}
\abs{\omega(\tfrac{m}{\ell})-\omega(\tfrac{m+2}{\ell})}^2
=O(1)
\end{align*}
as $\ell\to\infty$, where we have used the smoothness of $\omega$ at the last step.
\end{proof}

\subsection{The case $f(x)=x^k$}
\begin{lemma}
For any integer $k\geq1$, we have
$$
\lim_{\ell\to\infty}
\frac{1}{\ell+1}\Tr \bigl((C[\omega_\ell]M[\sigma]C[\omega_\ell]^*)^k\bigr)
=
\frac1{4\pi}\int_{-\pi}^{\pi}
\int_{-1}^1 \abs{\omega(\xi)}^{2k}(\sigma(\varphi))^kd\xi\,  d\varphi.
$$
\end{lemma}
\begin{proof}
We observe that by the cyclicity of trace,
$$
\Tr \bigl((C[\omega_\ell]M[\sigma]C[\omega_\ell]^*)^k\bigr)
=
\Tr \bigl((M[\sigma]C[\omega_\ell]^*C[\omega_\ell])^k\bigr)
$$
and $C[\omega_\ell]^*C[\omega_\ell]=C[\Omega_\ell]$, where
$$
\Omega_\ell(\varphi)
=
\sum_{\genfrac{}{}{0pt}{1}{m=-\ell+2}{\text{$m-\ell$ \rm even}}}^{\ell-2}
\abs{\omega(m/\ell)}^2e^{im\varphi}.
$$
Further, for $k=1$ (as in the second part of the proof of Theorem~\ref{thm.b1}, see Section~\ref{sec.3.10}) by a direct evaluation of trace we have
\begin{align*}
\Tr (M[\sigma]C[\Omega_\ell])
&=
\frac1{2\pi}\int_{-\pi}^\pi \sigma(\varphi)d\varphi
\sum_{\genfrac{}{}{0pt}{1}{m=-\ell+2}{\text{$m-\ell$ \rm even}}}^{\ell-2}
\abs{\omega(m/\ell)}^2
\\
&=\frac{1}{2\pi}\int_{-\pi}^\pi \sigma(\varphi)d\varphi\cdot
\frac{\ell}{2}\int_{-1}^1\abs{\omega(\xi)}^2d\xi+O(1),
\end{align*}
for $\ell\to\infty$.
If $k=2$, denoting the commutator $R=\bigl[C[\Omega_\ell],M[\sigma]\bigr]$,  we find
\begin{align*}
\Tr \bigl((M[\sigma]C[\Omega_\ell])^2\bigr)
=&
\Tr (M[\sigma]^2C[\Omega_\ell]^2)
+
\Tr(M[\sigma]R\ C[\Omega_\ell]).
\end{align*}
Here the first trace in the right hand side can be again directly evaluated:
\begin{align*}
\Tr (M[\sigma]^2C[\Omega_\ell]^2)
&=
\frac1{2\pi}\int_{-\pi}^\pi \sigma(\varphi)^2d\varphi
\sum_{\genfrac{}{}{0pt}{1}{m=-\ell+2}{\text{$m-\ell$ \rm even}}}^{\ell-2}
\abs{\omega(m/\ell)}^4
\\
&=\frac{1}{2\pi}\int_{-\pi}^\pi \sigma(\varphi)^2 d\varphi\cdot
\frac{\ell}{2}\int_{-1}^1\abs{\omega(\xi)}^{4}d\xi+O(1),
\end{align*}
and the second trace can be estimated as follows:
\begin{align*}
\abs{\Tr(M[\sigma]R\ C[\Omega_\ell])}
\leq
\norm{M[\sigma] R\ C[\Omega_\ell]}_{\Sch_1}
&\leq
\norm{M[\sigma]}
\norm{R}_{\Sch_2}
\norm{C[\Omega_\ell]}_{\Sch_2}
\\
&=O(1)O(1)O(\sqrt{\ell})=O(\sqrt{\ell}).
\end{align*}
Similarly for $k\geq3$ we find
$$
\Tr \bigl((M[\sigma]C[\Omega_\ell])^k\bigr)
=
\Tr (M[\sigma]^k C[\Omega_\ell]^k)
+
\text{error}
$$
where the trace in the right hand side can be directly evaluated and the error term is a combination of traces of commutator terms, which are all $O(\sqrt{\ell})$.
\end{proof}

\subsection{Concluding the proof: application of the Weierstrass approximation theorem}

For all $\ell$, we have the estimate
$$
\norm{C[\omega_\ell]M[\sigma]C[\omega_\ell]^*}\leq R, \qquad
R=\norm{\omega}^2_{L^\infty}\norm{\sigma}_{L^\infty}.
$$
It suffices to prove the statement for $f$ real-valued. We have already checked the statement for $f(x)=x$, and so, subtracting a linear term from $f$, we may assume that $f'(0)=0$.
Under this assumption, the function $f(y)/y^2$ is in $C^\infty$, and so by the Weierstrass approximation theorem for any $\eps>0$ we can find polynomials $p_+$ and $p_-$ vanishing at the origin such that
$$
p_-(y)\leq f(y)\leq p_+(y) \quad \text{ and }\quad
p_+(y)-p_-(y)\leq \eps y^2, \quad \text{ for all $\abs{y}\leq R$.}
$$
Then we have, using the previous lemma,
\begin{align*}
\limsup_{\ell\to\infty}
\frac{1}{\ell+1}\Tr f(C[\omega_\ell]M[\sigma]C[\omega_\ell]^*)
&\leq
\limsup_{\ell\to\infty}
\frac{1}{\ell+1}\Tr p_+(C[\omega_\ell]M[\sigma]C[\omega_\ell]^*)
\\
&=
\frac1{4\pi}\int_{-\pi}^{\pi}
\int_{-1}^1 p_+\bigl(\abs{\omega(\xi)}^2\sigma(\varphi)\bigr)d\xi\,  d\varphi
\end{align*}
and similarly
$$
\liminf_{\ell\to\infty}
\frac{1}{\ell+1}\Tr f(C[\omega_\ell]M[\sigma]C[\omega_\ell]^*)
\geq
\frac1{4\pi}\int_{-\pi}^{\pi}
\int_{-1}^1 p_-\bigl(\abs{\omega(\xi)}^2\sigma(\varphi)\bigr)d\xi\,  d\varphi.
$$
Furthermore, subtracting the right hand sides here, we find
\begin{align*}
\int_{-\pi}^{\pi}\int_{-1}^1 p_+\bigl(\abs{\omega(\xi)}^2\sigma(\varphi)\bigr)d\xi\,  d\varphi
&-
\int_{-\pi}^{\pi}\int_{-1}^1 p_-\bigl(\abs{\omega(\xi)}^2\sigma(\varphi)\bigr)d\xi\,  d\varphi
\\
&\leq
\eps
\int_{-\pi}^{\pi}\int_{-1}^1 \bigl(\abs{\omega(\xi)}^2\sigma(\varphi)\bigr)^2d\xi\,  d\varphi.
\end{align*}
Sending $\eps\to0$, we find that $\limsup$ and $\liminf$ above coincide, which gives the required statement.
\qed

\end{document}